\documentclass[reqno]{amsart}
\usepackage[english]{babel} 
\usepackage{amsmath,amsfonts, amsthm,amssymb,mathtools} 
\usepackage{mathrsfs,mathabx,bbold}    
\usepackage[alphabetic]{amsrefs}
\usepackage{hyperref} 
\usepackage{graphicx} 
\usepackage{framed}   
\usepackage{ifthen}   
\usepackage{lastpage} 
\usepackage{mathdots} 
\usepackage{enumerate}
\usepackage{tikz}     
\usepackage{verbatim} 
\usepackage{stmaryrd,mathrsfs,calrsfs,dutchcal,mathtools}
\usepackage{parskip}
\usepackage{abstract}
\usepackage{parskip}
\usepackage{enumitem}

\newcommand{\RR}{\mathbb{R}} 
\newcommand{\NN}{\mathbb{N}} 

\newcommand{\pp}{\mathcal{P}}

\newcommand{\bq}{\begin{question}}
	\newcommand{\eq}{\end{question}}
\def\ba#1\ea{\begin{align*}#1\end{align*}}

\newcommand{\powset}[2][]{\ifthenelse{\equal{#2}{}}{\mathcal{P}\left(#1\right)}{\mathcal{P}_{#1}\left(#2\right)}}

\newcommand{\dist}{\text{dist}}

\let\sumnonlimits\sum
\let\prodnonlimits\prod
\let\cupnonlimits\bigcup
\let\capnonlimits\bigcap
\renewcommand{\sum}{\sumnonlimits\limits}
\renewcommand{\prod}{\prodnonlimits\limits}
\renewcommand{\bigcup}{\cupnonlimits\limits}
\renewcommand{\bigcap}{\capnonlimits\limits}
\newcommand{\XX}{\mathfrak{X}}

\newcommand{\calA}{\mathcal{A}}

\newtheoremstyle{plainsl}
{8pt plus 2pt minus 4pt}
{8pt plus 2pt minus 4pt}
{\slshape}
{0pt}
{\bfseries}
{.}
{5pt plus 1pt minus 1pt}
{}

\theoremstyle{plainsl}
\newtheorem{theorem}{Theorem}[section]
\newtheorem{prop}[theorem]{Proposition}
\newtheorem{prop*}{Proposition}

\newtheorem{lemma}[theorem]{Lemma}
\newtheorem{corollary}[theorem]{Corollary}

\theoremstyle{definition}
\newtheorem{definition}[theorem]{Definition}
\newtheorem{example}[theorem]{Example}

\newtheorem{question}[theorem]{Question}

\theoremstyle{remark}
\newtheorem{remark}[theorem]{Remark}

\newcommand{\argmax}{\text{argmax}}

\newcommand{\nn}{\nonumber}
\def\ba#1\ea{\begin{align*}#1\end{align*}}

\newcommand{\Tr}{\text{Tr}}

\begin{document}

\title[Comparison theorem for viscosity solutions]{Comparison theorems for weak solutions of nonlinear maximally sub-elliptic PDEs}


\author{Gautam Neelakantan Memana}
\address{Van Vleck Hall, 213, 480 Lincoln Dr, Madison, WI 53706}
\curraddr{}
\email{neelakantanm@wisc.edu}
\thanks{}

\date{}

\dedicatory{}
	\maketitle
  \begin{abstract} 
We establish a comparison principle for viscosity subsolutions and supersolutions of a broad class of second-order quasilinear, maximally subelliptic PDEs on general manifolds. In fact, we prove the comparison theorem for a larger class of degenerate subelliptic PDEs. Our result strengthens a recent theorem of Manfredi–Mukherjee, which was established in the setting of Carnot groups. Our main aim is to highlight that maximal subellipticity allows one to obtain a comparison principle for weak solutions in close analogy with the classical elliptic theory.
\end{abstract}
 \numberwithin{equation}{section}
\section{Introduction}\label{Introduction}
Let $\mathcal{M}$ be a connected $C^{\infty}$ manifold of dimension $n$ with a smooth, strictly positive density $Vol$. Let $C^{\infty}(\mathcal{M}; T\mathcal{M})$ denote the smooth sections of the tangent bundle $T\mathcal{M}$. Let 
\begin{align*}
    \XX= \{X_1,..., X_m\} \subset C^{\infty}(\mathcal{M}; T\mathcal{M})
\end{align*}
be such that $\XX$ satisfies H\"ormander's bracket generating condition. i.e, 
\begin{align}\label{Hormander condition}
   & X_1(x), X_2(x), \cdots, X{m}(x), \cdots[X_i, X_j](x), \cdots, [X_i, [X_j, X_k]](x)\nonumber\\
    &\cdots, \cdots, \text{commutators of order $r$ evaluated at x},
\end{align}
for some finite $r\in \NN$. 
Let also let 
\begin{align}
    \XX\XX=\{X_iX_j\}_{1,j\in \{i,\cdots, m\}}
\end{align}
We will first describe a general fully nonlinear second order maximally subelliptic operators and then later on restrict our attention to quasilinear setting. The goal is to study weak solutions(viscosity sub/super solutions) for a PDE of the form 
\begin{align}\label{introduction: general partial differential equation}
    F(x,u, \XX u, \XX\XX u) =g(x),
\end{align}
 for a smooth function $F$ in all the variables(later on we will lower the regularity to just continuity).
 \begin{remark}
     Since $\{X_1,...,X_k\}$ satisfies H\"ormander's condition, every nonlinear PDE can be written in the form \eqref{introduction: general partial differential equation}. Hence, we are not making any assumptions on the form of the PDE.
 \end{remark}

\begin{definition}\label{introduction: definition of maximally subelliptic PDE}(Maximal subellipticity)
Let $x_0\in \mathcal{M}$ and $u:\mathcal{M}\to \RR$. We say that a PDE given by \eqref{introduction: general partial differential equation} for a smooth function $F(x, \zeta)$ is \textit{maximally subelliptic} at $(x_0,u)$ of degree $\kappa$ with respect to $W$ if there exists an open neighborhood $U\subseteq \mathcal{M}$ with $x_0\in U$ such that the linearized operator $\mathcal{P}_{u,x_0}$ defined by  
 \begin{align}\label{Introduction: linearized operator}
     \mathcal{P}_{u, x_0}v:= d_{\zeta} F(x_0,u(x_0), \XX u(x_0), \XX\XX u(x_0))\{\XX v, \XX\XX v\}
 \end{align}
is a linear \textit{maximally subelliptic} partial differential operator of degree $2$ with respect to $\XX$ on $U$,
i.e., for every relatively compact, open set $\Omega \Subset U$ there exists $C_{\Omega}\geq 0$ satisfying
 \begin{align}
     \sum_{j=1}^{m} \|X_{j}^{2}f\|_{L^2} \leq C_{\Omega} (\|\pp_{u,x_0} f\|_{L^2}+\|f\|_{L^2}), 
 \end{align}
for every $f\in C_{0}^{\infty}(\Omega)$.
\end{definition}
  Also, see \cite[Theorem 8.1.1]{BS} for equivalent characterizations of linear maximally subelliptic PDE. We also recommend the reader to see \cite{AMY22}, where the authors show equivalent representation theoretic characterization of linear maximally subelliptic operators by proving a conjecture of Helffer and Nourrigat \cite{AMY22}. For regularity theory for fully nonlinear maximally subelliptic PDEs see \cites{BS, NM24}.  
  
In this article, we consider quasi-linear equations of the form 
\begin{align}\label{Introduction: main equation}
    \sum_{i,j=1}^{m} A_{i,j}(\mathfrak{X}u)X_iX_ju= H(\mathfrak{X}u)\ \in \Omega,
\end{align}
where $\mathfrak{X}u= (X_1u,..., X_mu)$ is the sub-elliptic gradient, $\mathfrak{X}\mathfrak{X}u - (X_jX_iu)_{i,j}$ is the sub-elliptic second derivative, and $A_{i,j}, H:\RR^m\to \RR$ are continuous functions such that the $(m\times m)$ matrix $A(\xi)$ with entries $A_{i,j}(\xi)$ is symmetric and positive semi-definite. Hence, the equation \eqref{Introduction: main equation} can also be re-written as 
\begin{align}\label{Introduction: main equation re-written}
    \mathcal{L}u:= -\Tr(A(\mathfrak{X}u)\mathfrak{X}^*\mathfrak{X}u)+ H(\mathfrak{X}u)=0,
\end{align}
where $\mathfrak{X}^*\mathfrak{X}u$ is the symmetrized matrix with entries $\frac{1}{2}(X_i^*X_ju +X_j^*X_iu)$. Furthermore, we asssume that $A,H$ satisfy the following conditions of strict ellipticity and scaling;
\begin{enumerate}
   \item \label{non-degeneracy} $\mathcal{E}(\xi):= \langle A(\xi)\xi, \xi\rangle>0$, $\forall \xi \in \RR^m\backslash\{0\}$ \label{introduction: equation 1.6 of MM},
  \item $-\Tr (A(t\xi) X) \leq \frac{-\Tr(A(\xi)X)}{t\phi(1/t)}$, and $H(t\xi)\leq \frac{1}{\phi(1/t)} H(\xi)$ \label{introduction: equation 1.7 of MM},
\end{enumerate}
for any $t\geq 1, \xi\in \RR^m\backslash \{0\}$ and $X\in \RR^{m\times m}$ symmetric, where $\phi: (0,1]\to (0,1]$ is a strictly positive function. The equation \eqref{Introduction: main equation re-written} is said to be degenerate if $A(0)=0$. The reason why we see \eqref{introduction: equation 1.7 of MM} as a scaling condition is: if $\phi$ is a power function then $\xi\mapsto |\xi|\Tr(A(\xi)X)$ and $H(\xi)$ are doubling functions. 
 \begin{remark}
     Assume that $A(\xi)$ is strictly positive definite every $\xi\in \RR^\backslash \{0\}$. Also assume that $A$ and $H$ are smooth(differentiable is sufficient). Then the linearization of $\mathcal{L}$ at point $x_0$ and function $u$ as in \eqref{Introduction: linearized operator} is 
     \begin{align}
         \mathcal{P}_{u,x_0} v
=
-\operatorname{Tr}\!\Big(A\big(\XX u(x_0)\big)\,\XX^{*}\XX v\Big)
-\Big\langle D_{\xi}A\big(\XX u(x_0)\big)[\,\XX v\,],\,\XX^{*}\XX u(x_0)\Big\rangle_{F}
\;+\; D_{\xi}H\big(\XX u(x_0)\big)\cdot \XX v, 
     \end{align}
    where $D_{\xi}A(\cdot)[\cdot]$ is the derivative of $A$ with respect to its $\xi$-argument applied to $\XX v$, and $\langle \cdot, \cdot \rangle_{F}$ is the Frobenius matrix inner product and all the coefficients are frozen at $x_0$. Now, we refer the reader to \cite[Lemma 8.9.3]{BS} to see that \eqref{introduction: equation 1.6 of MM} implies that the above linearized operator is maximally subellitpic. Hence the operator from \eqref{Introduction: main equation re-written} contains nonlinear maximally subelliptic PDEs once you impose some smoothness assumption on the coefficients. 
 \end{remark}
 \begin{example}
     If $H\equiv 0$ and $A\equiv \mathbb{I}_m$ the identity matrix, then $\mathcal{L}$ corresponds to the sub-Laplacian $\Tr(\XX^*\XX)$.
    
 \end{example}
 \begin{example}
     If $H\equiv 0$ and $A(\xi)=\xi\otimes \xi$ we get the $\infty-$Laplacian 
    
        $$ \Delta_{\XX, \infty} = \langle \XX^*\XX u\XX u, \XX u \rangle.$$
         $\infty-$Laplacian has been extensively studied in the Euclidean setting; see the celebrated paper of \cite{Jensen} on uniqueness of viscosity solutions to Dirichlet boundary value problem.  
     \end{example}
     \begin{example}
         When $H\equiv 0$ and $A(\xi)=\mathbb{I}_{m} +(p-2) (\xi\otimes \xi)/|\xi|^2$ for $1<p<\infty$ we get the normalized $p-$Laplacian 
         \begin{align*}
             \Delta_{\XX, p}^N = \Tr(\XX^*\XX)+ (p-2) \frac{\Delta_{\XX, \infty} u}{|\XX u|^2}. 
         \end{align*}
     \end{example}
Even though our motivation to study PDE of the form \eqref{Introduction: main equation re-written} comes from trying to understand weak solutions of maximally subelliptic PDEs there is a rich history for such PDEs motivated from questions in sub-Riemannian geometry like Uhlenbeck-Uraltseva structure equation and mean curvature flow type equation to name some; we refer the reader to \cite{MM} for more history.

Now, we state the main theorem of the article which proves comparison principle for viscosity sub/super solutions for \eqref{Introduction: main equation re-written}.

 \begin{theorem}\label{Introduction: Theorem 1.1 of MM}
Let $M$ be a manifold and let $\Omega$ be an open and connected bounded domain of $M$. Let $\mathcal{L}$ be as in \eqref{Introduction: main equation re-written} satisfying condition \eqref{introduction: equation 1.6 of MM} and \eqref{introduction: equation 1.7 of MM}. If there exists $u,v\in C(\bar{\Omega})$ such that $\mathcal{L}u\leq 0\leq \mathcal{L}v$ in the viscosity sense in $\Omega$ and $u\leq v$ in $\partial \Omega$, then we have $u\leq v$ in $\Omega$. 
 \end{theorem}
\begin{remark}
    As in \cite{MM}, we remark that the solvability of the Dirichlet problem $\mathcal{L}u=0$ in $\Omega$ and $u=f$ in $\partial \Omega$ requires appropriate assumptions on the boundary; see \cites{ACJ, BB, CIL}. Once you have existence, Theorem \ref{Introduction: Theorem 1.1 of MM} will give uniqueness of the solution. 
\end{remark}
\subsection{Main Ideas} In \cite{BB}, Barles and Busca proved uniqueness result for Dirichlet problem for a domain $\Omega\subset\RR^n$ and $f\in C^{0,1}(\Omega)$ for a general class of degenerate elliptic equation that includes the $\infty-$Laplacian(Euclidean setting). In \cite{MM} extended the result of \cite{BB} to Carnot groups for sub-elliptic analgoues \eqref{Introduction: main equation re-written}. The extension in \cite{MM} required nontrival adaptations that includes strong maximum principle of \cite{BG}, Rademacher theorem for Carnot groups \cite{PP}, sup and inf-convolutions of \cite{WC}. 

As in the Euclidean case \cite{BB}, Manfredi and Mukherjee in \cite{MM} obtains comparison principle by approximating viscosity sub/super solutions with semi-convex/concave functions. Unlike in the Carnot group case we don't have the approximate semi-convex and semi-concave functions to be viscosity sub/super solutions in the entire domain. So, we prove a version of version of approximating the viscosity sub/super solutions from \cite[Proposition]{WC} that retains the viscosity sub/super solution property at selected points in the domain for a perturbed operator; see Section \ref{Approximating viscosity solutions}. We also remark that it seems unlikely that one could prove a general version of \cite[Proposition 3.3]{WC} on a manifold with given H\"ormander vector fields unlike in the case of homogeneous groups. 

To prove the comparison principle, we follow the structure of the proof from \cite{MM}. First the comparison principle is achieved with a non-degeneracy assumption(non-vanishing gradient); see Lemma \ref{Semi convex solutions: Lemma 3.4 from MM}, where a small perturbation with the strict ellipticity \eqref{introduction: equation 1.6 of MM} leads to strict sub-solutions which, by virtue of Jensen's lemma (Lemma \ref{Background: Jensen's lemma}) and Aleksandrov's theorem (Theorem \ref{Background: Aleksandrov's theorem}), are also classical sub-solutions at points of second order differentiability arbitrarily close to their maximal points. Then, we remove the non-degeneracy assumption in Proposition \ref{Proposition 3.5 of MM}, which is the key step to prove Theorem \ref{Introduction: Theorem 1.1 of MM}.

In in the proof of \cite[Proposition 3.5]{MM} (this is the counterpart of Proposition \ref{Proposition 3.5 of MM}) ran into a strange obstacle in the proof where they had to balance between Lipschitz continuity of translations of viscosity/sub-solution and left invariance preserved the fact that translations . Handling this situation was the main novelty of \cite{MM}. In this article we form translation of the viscosity sub/super solutions using exponential maps obtained from $\{X_1, \cdots, X_n\}$ and hence there is no hope for ``left invariance". The proofs in Section \ref{Comparison principle for semi convex/concave solutions} proceeds through proof by contradiction. The main novelty of this article to get contradiction at the shifted points(using exponential map) that is obtained through careful perturbations of the operator $\mathcal{L}$ as constructed in Section \ref{Approximating viscosity solutions}. Other new ideas in the proof include Gr\"onwall type bounds to obtain Lipschitz continuity in \eqref{semi-convex solutions: equation 3.19 of MM}, Rademacher theorem for sub-Riemannian manifold by \cite{CJ}, symmetrized smoothened Carnot-Carath\'eodory metric of \cite{MR1882665}.

  \subsection*{Acknowledgements} The author wishes to thank Brian Street for all the helpful discussions during the preparation of this manuscript. The author was partially supported by NSF DMS 2153069. 
 \section{Background}
 In this section we will introduce the notations and the background required for the rest of the article. We will try to maintain the same notation as \cite{MM} for the ease of comparison with \cite{MM} for the reader. The standard Euclidean inner product on $\RR^n$ is denoted by $\langle \cdot, \cdot\rangle$, the Euclidean vector fields are denoted as $\partial_{x_i}$ for $i\in \{1 ,.\cdots, n\}$ and $\nabla u=(\partial_{x_1}u, \cdots, \partial_{x_n}u)$ is the gradient, $DF$ is the Jacobian matrix for $F:\Omega\to \RR^n$ and $D^2 u=(D\nabla u)= (\partial x_i\partial x_j u)_{i,j}$ is the Hessian. 
\subsection{H\"ormander Vector Fields}\label{Background: Hormander vector fields}
Let $\{X_1,..., X_m\}$ be a set of $C^{\infty}$ vector fields in a neighborhood of $\Omega\subset \bar{\Omega}\subset\subset\mathcal{M}$. From now on, we will denote the vector $(X_1,..., X_m)$ as $\mathfrak{X}$. Let $r$ be such that the span of commutators of up to $r$ spans the tangent at each point of $\Omega$ as in \eqref{Hormander condition}. Since $\bar{\Omega}$ is compact, without loss of generality we will assume that $r$ is finite and minimal. 
 \subsection{Properties of $\mathcal{L}$} We will see the properties of $\mathcal{L}$ that we can derive from the structure conditions we imposed in \eqref{introduction: equation 1.6 of MM} and \eqref{introduction: equation 1.7 of MM}. For symmetric square matrices $A, B\in \RR^{k\times k}$, we shall denote 
 \begin{align}
     A\leq B \ \iff \ \langle A\xi, \xi \rangle\leq \langle B\xi, \xi\rangle\ \forall \ \xi\in \RR^k.
 \end{align}
 The Frobenius inner product of matrices $A, B\in \RR^{k\times k}$ is given by $\Tr(A^TB)$ and the Frobenius norm $\|A\|=\sqrt{\Tr(A^TA)}$. For non-negative matrices $A, Z\geq 0$, it is not difficult to see that $\Tr(AZ)=\|\sqrt{Z}\sqrt{A}\|^2\geq 0$, and hence, we have 
 \begin{align}\label{eq 2.2 of MM}
    \Tr(AY)\leq \Tr(AX), \quad \ A\geq 0, Y\leq X.
 \end{align}
 Now, observe the following. Take $X=\xi\otimes \xi$ in \eqref{introduction: equation 1.7 of MM} for any $\xi\in \RR^m$, we get
 \begin{align}\label{2.3 of MM}
     -\langle A(t\xi), \xi\rangle\leq \frac{-1}{t\phi(t)} \langle A(\xi)\xi, \xi\rangle, 
 \end{align}
 for any $t\geq 1$. Henceforth, the growth of $\mathcal{E}$ as in \eqref{introduction: equation 1.6 of MM} can obtained when the gradient variable is away from the origin in $\RR^m$. Precisely, for any $\theta>0$ and $\xi\in \RR^m$ with $|\xi|\geq 0$, we can use \eqref{2.3 of MM} with $\xi\mapsto \theta\xi/|\xi|$ and $t=|\xi|/\theta$ to obtain 
 \begin{align}\label{2.4 of MM}
     \mathcal{E}(\xi)\geq \frac{|\xi|a_{\theta}}{\theta\phi(\theta/|\xi|)}\geq \frac{a_{\theta}}{\phi(\theta/ |\xi|)}, 
 \end{align}
 where $a_{\theta}:=\inf_{|\zeta|}\mathcal{E}(\zeta)>0$ from the ellipticity condition \eqref{introduction: equation 1.6 of MM}, for all $|\xi|\geq \theta>0$. 

For smooth H\"ormander vector fields $\{X_1, \cdots, X_m\}$ on $\Omega$ and a function $u:\Omega\to \RR$, the maximally sub-ellitpic gradient and the second derivative matrices are dentoed as 
\begin{align*}
    \XX u= (X_1 u, \cdots, X_m u)\ \text{and}\ \XX\XX u=(X_jX_i u)_{i,j}.
\end{align*}
Ofcourse, one could define these even if the vector fields did not satisfy the H\"ormander condition. But, we wouldn't have the maximal subelliptcity of the gradient or the following symmetrized operator
\begin{align*}
    \XX^*\XX u=\frac{1}{2}(X_i^*X_j u +X_j^*X_i u)_{1,j},
\end{align*}
where $X_j^*$ is the adjoint of $X_j$ with respect to $L^2(\Omega).$
The divergence of $F=(f_1, \cdots, f_m)$ with respect to the vector fields is defined by $\text{div}_{\XX}(F)=\sum_{j=1}^{m}X_{j}^* f_j$. 

It is not hard to see that there exists $\sigma; \Omega\to \RR^{n\times m}$, written as 
\begin{align*}
    \sigma(x)=(\sigma^j_i(x))_{i,j}=[\sigma^1(x), \cdots, \sigma^m(x)],
\end{align*}
with $\sigma^j:\Omega\to \RR^n$, such that $X_j=\sigma^j(x)\nabla$. Since the vector fields are smooth, the mapping $x\mapsto \sigma(x)$ is smooth and H\"ormander's bracket generating condition implies $\sigma(x)\neq 0$ for all $x\in \Omega$. For any $u: \Omega\to \RR$, set 
\begin{align}
    \XX u=\sigma(x)^T \nabla u \quad \text{and}\quad \XX^*\XX u= \sigma(x)^T D^2 u\sigma(x)+M(x, \nabla u),
\end{align}
where $M(x, \zeta)\in \RR^{m\times m}$ is linear in $\zeta$ and smooth in $x$, and is built from the derivatives of $\sigma$, i.em .
\begin{align}
    M(x, \zeta)_{i,j} = \frac{1}{2} (\langle D \sigma^j(x)\sigma^i(x), \zeta\rangle +\langle D\sigma^i(x)\sigma^j(x), \zeta\rangle ). 
\end{align}
With this, we can re-write the operator $$\mathcal{L}u= -\Tr(A(\XX u)\XX^*\XX u)+H(\XX u)$$
as 
\begin{align}\label{euclidean form of operator}
    -\Tr(A(\sigma^T \nabla u)[\sigma^T D^2 u \sigma +M(x, \nabla u)]) + H(\sigma^T \nabla u).
\end{align}
Observe that we didn't need $X_i^*=-X_i$(as in the case of left invariant vector fields on Carnot groups) for the non-divergence form of our operator.

Since the domain $\Omega \subset \RR^n$ bounded, for any function $f:\Omega\to \RR$, we shall denote the set of maximum and minimum points as 
\begin{align}
    \argmax_{\Omega}(f):= \{x\in \Omega:f(x)=\max_{\Omega}f\}, \quad     \text{argmin}_{\Omega}(f):= \{x\in \Omega:f(x)=\min_{\Omega}f\}.
\end{align}
If the function does not have any local maxima or minima in $\Omega$, then the respective sets of the above are empty. Sometimes the subscript is dropped when the context for the corresponding domain of argmax or armin is clear. 
 \subsection{Viscosity solutions}
The goal of this section is to define viscosity sub/super-solutions for the partial differential operator we are concerned about and then state a strong maximum principle due to Bardi-Goffi \cite{BG} for a class of fully nonlinear subelliptic operators(\eqref{Introduction: main equation re-written} is included in this). For classical theory of viscosity solutions see \cites{ACJ, CIL}. We will describe viscosity sub/super solutions to \eqref{Introduction: main equation re-written}. Observe that $\mathcal{L}u=-\Tr(A(\mathfrak{X}u)\mathfrak{X}^*\mathfrak{X}u)+ H(\mathfrak{X}u)$ is of the form 
\begin{align}\label{general fully nonlinear second order}
    G(x, u, \XX u, \XX^*\XX u)=0,
\end{align}
for some function $G: \bar{\Omega}\times \RR\times \RR^m\times \RR^{m\times m}\to \RR$. Since we have \eqref{euclidean form of operator} form of $\mathcal{L}$ we define the notion of viscosity in the same way as in Euclidean space. Let us denote the classes of upper and and lower semi-continuous functions as USC and LSC. For any $w:\Omega\to \RR$ and $x\in \Omega$, let us denote the class of test functions $\mathcal{A}_{x}^{\pm}(\Omega)(w, \Omega)$ as 
\begin{align}
    \mathcal{A}_{x}^{+}(w, \Omega)=\{ \phi\in C^2(\Omega): x\in \argmax_{\Omega}(w-\phi), \nabla\phi \neq 0\}, 
\end{align}
and $\mathcal{A}_{x}^{-}(w, \Omega)$ defined similarly, replacing $\argmax$ with $\text{argmin}$. Observe that for any invertible $\Theta:\Omega\to \RR$, we have $\phi\in \mathcal{A}_{x}^{\pm}(w, \Omega)$ if and only if $\phi\circ \Theta^{-1}\in \mathcal{A}_{\Theta(x)}^{\pm}(w\circ \Theta^{-1}, \Theta(\Omega))$.
\begin{definition}
    For equation \eqref{general fully nonlinear second order}, $u\in \text{USC}(\Omega)$(resp $u\in \text{LSC}(\Omega)$) is called a viscosity subsolution(resp. supersolution) at $x\in \Omega$ if for every $\phi\in \mathcal{A}_{x}^{+}(w, \Omega)$(resp. $\mathcal{A}_{x}^{-1}(w, \Omega)$), we have
    \begin{align*}
        G(x, \phi(x), \XX\phi(x), \XX^*\XX\phi(x))\leq 0\quad (\text{resp. $\geq 0$}),
    \end{align*}
    which is referred as $G(x, u, \XX u, \XX^*\XX u)\leq 0$ (resp. $\geq0$) in the viscosity sense. If both of the above inequalities hold simultaneously for respective test functions in $\mathcal{A}_{x}^+(u, \Omega)$ and $\mathcal{A}_{x}^-(u, \Omega)$, then $u$ is called a viscosity solution of equation \ref{general fully nonlinear second order}.
\end{definition}

Thus, the viscosity sub/super solution of $u$ \eqref{Introduction: main equation re-written} at $x\in \Omega$ implies $\mathcal{L}\phi(x)\leq 0$(resp. $\geq 0$) for all $\phi\in \mathcal{A}_{x}^+(u, \Omega)$(resp. $\mathcal{A}_x^{-}(u, \Omega)$). 

Following \cite{MM},  we need the following strong maximum principle due to Bardi-Goffi\cite{BG} for Proposition \ref{Proposition 3.5 of MM}
\begin{theorem}(Strong Maximum principle \cite{BG})\label{stron maximum principle} Given smooth vector fields $X_1, \cdots X_m$ satisfying H\"ormander's bracket generating condition \eqref{Hormander condition}, if a function $G:\bar{\Omega}\times \RR\times \RR^m\times \RR^{m\times}\to \RR$ satisfies the following:
\begin{enumerate}
    \item $G$ is lower smincontinuous and for all $r\leq s$ and symmetric matrices $Y\leq X$, 
    \begin{align*}
        G(x, r, \xi, X)\leq G(x, s, \xi, Y);
    \end{align*}
    \item there exists $\phi: (0,1]\to (0,\infty)$ such that for all $\lambda \in (0,1], x\in \Omega, r\in [-1, 0], \xi\in \RR^{m}\backslash\{0\}$ and symmetric $X\in \RR^{m\times m},$ we have 
    \begin{align*}
        G(x, \lambda r, \lambda \xi, \lambda X)\geq \phi(\lambda)G(x, r, \xi, X);
    \end{align*}
    \item for all $x\in \Omega, \xi\in \RR^m\backslash \{0\}, X\in \RR^{m\times m}$, the following ellipticity condition holds, 
    \begin{align*}
        \sup_{\gamma>0}G(x, 0, \xi, X-\gamma \xi\otimes \xi)>0;
    \end{align*}
    then, any viscosity sub-solution(resp. super-solution) of the equation $G(x, u, \XX u, \XX^*\XX u)$ that attains a non-negative(resp. non-positive) maximum (resp. minimum) in $\Omega$ is constant. 
\end{enumerate}
\end{theorem}
\eqref{Introduction: main equation re-written} is an example of Theorem \ref{stron maximum principle}, which we state in the following corollary.
\begin{corollary}\label{Background: Corolllary 2.3 of MM}\cite[Corollary 2.3]{MM}
    Given the equation \eqref{Introduction: main equation re-written} with $A:\RR^m\to \RR^{m\times m}$ and $H: \RR^m \to \RR$ satisfying \eqref{introduction: equation 1.6 of MM} and \eqref{introduction: equation 1.7 of MM}, any viscosity sub-solution (resp. super-solution) that attains a on-negative(resp. non-positive) maximum(reps. minimum) in $\Omega$ is constant. 
\end{corollary}
\begin{proof}
    It is not hard to see that $G(x, r, \xi, X)=-\Tr(A(\xi)X)+H(\xi)$ satisfies the hypotheses of Theorem \ref{stron maximum principle}. Since $A(\xi)$ is symmetric and positive semi-definite, \eqref{eq 2.2 of MM} implies (1). Taking $\lambda \mapsto \frac{1}{\lambda}$ and $\xi/\lambda\to \xi$ on the scaling condition \eqref{introduction: equation 1.7 of MM} leads to 
    \begin{align*}
        -\lambda \Tr(A(\lambda \xi)X)+ H(\lambda \xi) \geq \phi(\lambda)\left(-Tr(A(\xi)X)+H(\xi)\right),
    \end{align*}
    and hence we have (2). The ellipticity condition \eqref{introduction: equation 1.6 of MM} leads to (3) because 
    \begin{align*}
        G(x, 0, \xi, X-\gamma\xi\otimes \xi)=\gamma\langle A(\xi)\xi, \xi\rangle-\Tr(A(\xi)X) +H(\xi)>0
    \end{align*}
    for all $\xi\in \RR^m\backslash\{0\}$, whenever $\gamma> (\Tr(A(\xi)X)-H(\xi))/\langle A(\xi)\xi, \xi\rangle$. Hence we are done.

\end{proof}
\subsection{Semi-concave and semi-convex functions}
Since we are dealing with weak solutions and they need not have enough regularity. So, we approximate our viscosity sub/super-solutions by functions that have more regularity. In particular, by semi-convex and semi-concave functions.  

In this section we assume that $\bar{\Omega}$ is contained in a coordinate patch of the manifold $\mathcal{M}$(since later on we will only be concered about local semi-convexity and local semi-concavity; see Section \ref{Approximating viscosity solutions})
\begin{definition}
    A function $w\in C(\bar{\Omega})$ is called semi-convex if there exists $\Lambda>$ such that $x\mapsto w(x)+\frac{1}{2}\Lambda |x|^2$ is convex; $w$ is called semi-concave if $-w$ is semi-convex. See \cite{BB} for the proof of the following:
    \begin{enumerate}
        \item (Differentiability at maximal points) Let $u,v\in C(\bar{\Omega})$ be a semi-convex and semi-concave function respectively. Then, both $u, v$ are differentiable at points in $\argmax_{\Omega}(u-v)$.
        \item (Partial continuity of the gradient) Let $w\in C(\bar{\Omega})$ be a semi-convex or semi-concave function that is differentiable at $x\in \Omega$ and at the points of a sequence $\{x_k\}_{k\in \NN}$ such that $x_k\to x$ as $k\to \infty$. Then, we have $\nabla w(x_k)\to \nabla w(x)$.
    \end{enumerate}
It should  be mentioned that if $w\in C^2(\Omega)$ then semi-convexity is equivalent to $-\Lambda \mathbb{I}\leq D^2w$. Therefore, for second order differentiability the classical theorem due to Aleksandrov for convex functions can be stated for semi-convex/semi-concave functions as the following; see  \cite[Theorem A.2]{CIL}
\end{definition}
\begin{theorem}\label{Background: Aleksandrov's theorem}(\text{Aleksandrov}) If $w: \RR^n\to \RR$ is semi-convex, it is twice differentiable a.e.

\end{theorem}
Even if $w$ is twice differentiable almost everywhere, it does not guarantee the twice differentiability at maximal points icen the set $\argmax(w)$ can be of measure zero and therefore, may remain entirely in the complement of the twice differentiable subset. The follwoing lemma shows that linear pertubrations can be chosen without hampering the second-order differential such that points arbitrarily close to maximal points of $w$ are within the twice differentiability subset and are also themselves the maximal points of the perturbations; see \cite[Lemma A.3]{CIL} for details. 
    \begin{lemma}(\text{Jensen's lemma})\label{Background: Jensen's lemma} Let $w: \RR^n\to \RR$ be semi-convex, $\hat{x}\in \argmax(w)$ be an arbitrary maximal point and $w_p(x)= w(x)+ p\cdot x$ for any $p\in \RR^n$. Then, for any $r,\delta>0$, the set 
    \begin{align}
        K_{r,\delta}(\hat{x})= \bigcup_{|p|<\delta} B_{r}(\hat{x})\cap \argmax(w_p)
    \end{align}
       is of positive Lebesgue measure.  
    \end{lemma}

From Aleksadrov's theorem $D^2w$ exists a.e. in $K_{r, \delta}(\hat{x})$ for any small enough $r, \delta>0$ even though it may not exist at $\hat{x}\in \argmax(w)$. Also, note that $D^2 w_{p}=D^2w_{p'}$ wherever it exists for any $p, p'\in \RR^n$. Therefore, we can always select a close enough point $z\in \RR^n$ with $|z-\hat{x}|$ such that $D^2w$ exists at $z$ and a small enough perturbation $w_p$ such that $|p|<\delta$ and $z\in \argmax(w_p)$; thus $z$ is a maximal point of $w_p$ with both first and second derivatives. 

\subsection{Carnot-Carathéodory Geometry}
We need the following smooth version of the Carnot-Carath\'eodory metric from \cite{MR1882665}. Let $\rho$ denote the Carnot-Carathéodory defined by the vector fields $\{X_1,\cdots, X_m\}$. Then, we have the following two technical lemmas. First, we need a smoothened version of the Carnot-Carathéodory metric that is due to Nagel and Stein \cite{MR1882665}. 
\begin{prop}\cite{MR1882665}\label{Approximation viscosity solutions: smooth metric}
    There exists a function $d: M\times M\to \RR^+$ such that:
    \begin{align*}
        d(x,y) \approx \rho(x,y)
    \end{align*}
    and for $x\neq y$, 
    \begin{align}\label{approximating viscosity solutions: derivative estimate for d}
        \left|D_{x}^{\alpha}D^{\beta}_{y} d(x,y)\right|\lesssim d(x,y)^{1-|\alpha|-|\beta|}.
    \end{align}
    By replacing $d(x,y)$ with $d(x,y)+d(y,x)$ we may assume that $d(x,y)=d(y,x)$. By multiplying $d$ by a fixed constant, we may also assume: 
    \begin{align*}
        \underset{x\neq y}{\sup}\ \sum_{|\alpha|=1} \left|D_{y}^{\alpha}d(x,y)\right|\leq 1. 
    \end{align*}
\end{prop}
Next, lemma gives an estimate for the smoothened metric $d$ when compared to the Euclidean metric locally in terms of $r$, from H\"ormander's bracker generating condition \eqref{Hormander condition}.
\begin{lemma}\cite[Proposition 1.1]{NSW}\label{Background: NSW lemma}
    Let $r$ be as in Section \ref{Background: Hormander vector fields}. Then, there exists constants $C_1, C_2$ such that for all $x\in \bar{\Omega}$ and $h,l\in \RR^m$ such that $\exp_{x}(h\cdot\XX), \exp_{x}(l\cdot \XX)\in \bar{\Omega}$, 
    \begin{align}\label{bound on cc metric after exp flow}
        C_1|h-l| \leq d(\exp_{x}(h\cdot \XX), \exp_{x}(l\cdot\XX))\leq C_2|h-l|^{1/r}.
    \end{align}
\end{lemma}
Now, we will state Rademacher's theorem for a general Carnot-Carathéodory geometry due to Cheeger \cite{CJ}(this is the variant we use for our result instead of the result of Pansu \cite{PP} Carnot groups, which is used in \cite{MM}). 
\begin{lemma}\cite{CJ}\label{Cheeger's result on a.e differentiability}
    Let $f:U\subset M\to \RR$ is a Lipschitz function with respect to distance function $\rho$, i.e $|f(x)-f(y)|\leq c d(x,y)$, then $f$ is differentiable Lebesgue $a.e.$ $x\in U$. 
\end{lemma}
\section{Approximating viscosity solutions}\label{Approximating viscosity solutions}

\numberwithin{equation}{section}
In this section we will see how to approximate any viscosity subsolution (supersolution) with semiconvex (semiconcave) subsolutions (supersolutions) for a perturbed operator. We use ideas similar to \cites{JLS, WC}. 

Let $\Omega\subset M$ be a bounded domain and let $d$ be as in Proposition \ref{Approximation viscosity solutions: smooth metric}. For any $\epsilon>0$, define
\begin{align*}
    \Omega_{\epsilon}= \{ x\in \Omega| \underset{y\in M\backslash \Omega}{\inf} d^2(x,y) \geq \epsilon\}
\end{align*}
\begin{definition}
    For any $u\in C(\bar{\Omega})$ and $\epsilon>0$, the $\sup$ convolution $u^{\epsilon}$ of $u$ is defined by 
    \begin{align}
        u^{\epsilon}(x)=\underset{y\in \bar{\Omega}}{\sup}\ \{ u(y)- \frac{1}{2\epsilon} d^2(x,y)\}, \quad \forall x\in \Omega.
    \end{align}
    Similarly, the $\inf$ convolution $u_{\epsilon}$ of $u\in C(\bar{\Omega})$ is defined by 
  \begin{align}
        u_{\epsilon}(x)=\underset{y\in \bar{\Omega}}{\inf}\ \{ u(y)- \frac{1}{2\epsilon} d^2(x,y)\}, \quad \forall x\in \Omega.
    \end{align}
\end{definition}

\begin{prop}\label{approximating by convex solutions}
    For any $u, v\in C(\bar{\Omega})$, denote $R_0=2\max\{\|u\|_{L^{\infty}(\Omega)}, \|v\|_{L^{\infty}(\Omega)}\}$. Then, for $\epsilon>0$ small enough we have
    \begin{enumerate}
        \item $u^{\epsilon}$ is semiconvex in $U$ and $v_{\epsilon}$ is semiconcave in $U$.
        \item $u^{\epsilon}$ (or $v_{\epsilon}$, resp) is monotonically non-decreasing (or non-increasing, resp.) w.r.t $\epsilon$, and converges unifomly to u in $\Omega_{(1+4R_0)\epsilon}$.
        \item\label{point 3 approximating by viscosity solution} Let $\{x_1, \cdots ,x_k\}$ be a set of points in $\Omega_{(1+4R)\epsilon}$ . Then there exists for $h\in \RR^m$ small enough such that there is a local perturbation of $\mathcal{L}$ near $\{x_1, \cdots, x_k\}$ given by $\mathcal{L}^\epsilon$ such that $\mathcal{L}^\epsilon$ satisfies the same conditions \eqref{introduction: equation 1.6 of MM} and \eqref{introduction: equation 1.7 of MM} imposed on $\mathcal{L}$ and $\mathcal{L}^\epsilon\to \mathcal{L}$ locally uniformly when tested on $C^2$ functions. If $\mathcal{L}u\leq \mathcal{L}v$ in $U$ at $\{x_1, \cdots, x_n\}$ in the viscosity sense then we have $\mathcal{L}^{\epsilon}u^{\epsilon}\leq 0\leq\mathcal{L}^{\epsilon}v_{\epsilon}$ in viscosity sense at $\{\exp_{x_1}(h\cdot \XX), \cdots, \exp_{x_k}(h\cdot \XX)\}$.
    \end{enumerate}
\end{prop}
\begin{proof}
Since the proof for $v_{\epsilon}$ is parallel to that for $u^{\epsilon}$, we only consider $u^{\epsilon}$. In this proof, the distance $d$ will be the smoothened symmetrized metric from Proposition \ref{Approximation viscosity solutions: smooth metric}.

  (1) By Leibniz rule we see that 
    \begin{equation*}
       \| D^2_{x}\{d^2(x,y)\}\|_{L^{\infty}(\bar{\Omega}\times \bar{\Omega})} \lesssim \|D_{x}\{d(x,y)\}\|_{L^{\infty}(\bar{\Omega}\times \bar{\Omega})}^2+\|d(x,y) D_x^2\{d(x,y)\}\|_{L^{\infty}(\bar{\Omega}\times \bar{\Omega})}
    \end{equation*}
    Now, since the diagonal of $\bar{\Omega}\times \bar{\Omega}$ has zero measure, and applying \eqref{approximating viscosity solutions: derivative estimate for d} the above can be bounded by
    \begin{align*}
        \lesssim \| d(x,y) ^{1-1}\|_{L^{\infty}(\bar{\Omega}\times \bar{\Omega})} + \|d(x,y)^{1+1-2}\|_{L^{\infty}(\bar{\Omega}\times \bar{\Omega})} <\infty,    \end{align*}
        where we also used the fact that $\bar{\Omega}\times \bar{\Omega}$ is bounded. Hence we define 
        \begin{equation*}
            C_{d}(\Omega)=\| D^2_{x}\{d^2(x,y)\}\|_{L^{\infty}(\bar{\Omega}\times \bar{\Omega})}.
        \end{equation*}
        Therefore, for any $y\in \bar{\Omega}$, the function 
        \begin{equation*}
            \bar{u}^{\epsilon}_{y}:= u(y) -\frac{1}{2\epsilon} d(x,y)^2 + \frac{C_{d}(\Omega)}{2\epsilon} |x|^2, \ \forall x\in \Omega,
        \end{equation*} 
        has nonnegative hessian and is convex. Since the supremum of a family of convex functions is convex, we have
        \begin{equation*}
            u_{\epsilon} (x) +\frac{C_d(\Omega)}{2\epsilon} |x|^2 = \underset{y\in \bar{\Omega}}{\sup}\ 
 \bar{u}^{\epsilon}_y (x)        \end{equation*}
 is convex so that $u_{\epsilon}$ is semi-convex. 

(2) It is easy to see that for any $\epsilon_1<\epsilon_2$, $u^{\epsilon_1}(x) \leq u^{\epsilon_2}(x)$ and $u(x)\leq u^{\epsilon}(x)\leq R_0$ for any $\epsilon>0$ and for any $x\in \Omega$, Observe that for any $x\in \Omega$ we have 
 \begin{equation*}
     u^{\epsilon}(x):= \underset{\bar{\Omega} \cap \{d(x,y)^2\leq 4R_0\epsilon\}}{\sup} (u(y)- \frac{1}{2\epsilon} d(x,y)^2).
 \end{equation*}
 Therefore, for any $x\in \Omega_{(1+4R_0)\epsilon}$, $u^{\epsilon}(x)$ is attained at a point $y\in \Omega$. To see $u^{\epsilon}\to u$ uniformly on $\Omega_{(1+4R_0)\epsilon}$, we observe that if $u^{\epsilon}(x)$ is attained by $x_\epsilon$, then 
 \begin{equation*}
     u^{\frac{\epsilon}{2}}(x) \leq u^{\epsilon}(x)= u(x_{\epsilon}) -\frac{1}{\epsilon} d(x,x_{\epsilon})^2 = u^{\epsilon}(x)- \frac{1}{2\epsilon} d(x, x_{\epsilon})^2, 
 \end{equation*}
 and hence 
 \begin{align*}
     \frac{1}{\epsilon}d(x, x_{\epsilon})^2\leq 2(u^{\epsilon}(x)-u^{\epsilon/2}(x)).
 \end{align*}
Since $u^{\epsilon}(x)-u^{\epsilon/2}(x)\to 0$ as $\epsilon \to 0$ we get 
\begin{equation*}
    \underset{\epsilon\to 0}{\lim}\ \frac{1}{\epsilon} d(x,x_{\epsilon})^2=0,
\end{equation*}
so that $x_{\epsilon}\to x$ and $\underset{\epsilon \to 0}{\lim}\ u^{\epsilon}(x)= u(x)$. Since 
\begin{equation*}
    |u^{\epsilon}(x_1)-u^{\epsilon}(x_2)| \leq |u(x_1)-u(x_2)|, \ \forall x_1, x_2 \in \Omega,
\end{equation*}

(3) Let $x_0\in \{x_1, \cdots, x_k\}$ represent a generic point. Let $\phi \in C^{2}(\Omega_{(1+4R_0)\epsilon})$ be such that 

\begin{equation*}
    u^{\epsilon}(x_0)-\phi(x_0)\geq u^{\epsilon}(x)-\phi(x), \ \forall x\in \Omega_{(1+4R_0)\epsilon}.
\end{equation*}
Here we will rewrite the partial differential operator $\mathcal{L}$ locally using exponential coordinates. $u^{\epsilon}\to u$ uniformly on $\Omega_{(1+4R_0)\epsilon}$. Similar to the argument in (2) we see that there exists $y_0\in \Omega$ such that 
\begin{equation*}
    u^{\epsilon}(x_0)= u(y_0)-\frac{1}{2\epsilon} d(x_0, y_0)^2.
\end{equation*}
Therefore we have
\begin{align*}
   & u(\exp_{x_0}(t_0\XX))-\frac{1}{2\epsilon} d(x_0,\exp_{x_0}(t_0\XX))^2 - \phi(x_0)\\ &\geq u(\exp_{x_0}(t\XX))-\frac{1}{2\epsilon}d(\exp_{x_0}(s\XX),\exp_{x_0}(t\XX))-\phi(\exp_{x_0}(s\XX)), \ \forall s, t\in U.
\end{align*}
Now, for $t$ near $t_0$, say $|t-t_0|<\epsilon/2$ we choose $s= t-t_0\in U$ and hence the above inequality becomes
\begin{align*}
    u(\exp_{x_0}(t_0\XX))- \phi(\exp_{x_0}((t-t_0)\XX))\geq u(\exp_{x_0}(t\XX))-\phi(\exp_{x_0}((t-t_0)\XX)).
\end{align*}
Let $\tilde{\phi}(\exp_{x_0}(t\XX)):= \phi(\exp_{x_0}((t-t_0)\XX))$ for $t$ near $t_0$. Then, $\tilde{\phi}$ touches $u$ from above at $t=t_0$ so that we have 
\begin{equation*}
    \mathcal{L} (\tilde{\phi})(t_0)\leq 0. 
\end{equation*}
Unlike the proof in \cite[Proposition 3.3]{WC}, we do not have any left invariance here. Hence, we won't be able to prove that $u^{\epsilon}$ is a viscosity subsolution at $\exp_{x_0}(t_0\cdot \XX)$. Hence, we will construct a perturbation of the operator $\mathcal{L}$ for which $u^{\epsilon}$ is a viscosity subsolution at $t-t_0$.
To do this, we will first pullback the operator $\mathcal{L}$ in a neighborhood of $x_0$. \\

We are ready to define perturbations of $\mathcal{L}$ which will be a major theme in this article. The idea is to locally pull back the vector fields near the given points $\{x_1, \cdots, x_n\}$ by the exponential map, then perturb it and then push it forward. Then, we patch together these operators using bump functions.

We have distinct points $\{x_1, \cdots, x_k\}$. For each $j$, choose open set 
\begin{align*}
    B_j^{\epsilon'}\Subset B_j^{\epsilon}\Subset \Omega_{(1+4R)\epsilon}, \quad B_i^\epsilon\cap B_j^\epsilon=\emptyset (i\neq j)
\end{align*}
and a cutoff $\eta_{j}\in C_{c}^{\infty}(B_j^\epsilon)$ with $\eta_{j}\equiv 1$ on $B_{j}^{\epsilon'}$ and we assume that $B_j^{\epsilon}\to \{x_j\}$ uniformly as $\epsilon \to 0$.  \\

Let $\Phi_j(t)=\exp_{x_j}(t\cdot \XX)$ for $|t|$ sufficiently small(i.e, assume $\exp_{x_j}(t\cdot \XX)\subset B_j^\epsilon$) and set
\begin{align}\label{pullback of vector  field}
    \hat{u}_j:=u\circ \Phi_j, \ \hat{X}_i^j:= \Phi_j^* X_i= \left(d\Phi_j^{-1}\right)X_i\circ \Phi_j.
\end{align}
Then for every smooth function $f$ we have 
\begin{align*}
    \hat{X}^j_i(f\circ \Phi_j)= (X_if) \circ \Phi_j, \quad \hat{X}^j_i\hat{X}^j_k(f)=(X_iX_kf)\circ \Phi_j, \quad  i,k \in \{1, \cdots, m\}.
\end{align*}
Therefore the horizontal gradient and the Hessian pull back becomes
\begin{align*}
    \hat{\XX}^ju:= (\XX u)\circ \Phi_j, \quad \hat{\XX}^{j*}\hat{\XX}^j\hat{u}:=(\XX^*\XX u)\circ \Phi_j .
\end{align*}
Hence $\mathcal{L}$ under the pull back becomes 
\begin{align}
   \hat{\mathcal{L}}^j\hat{u}^j(t)= -\Tr\left(A(\hat{\XX}^j\hat{u}^j(t))\hat{\XX}^{j*}\hat{\XX}^j\hat{u}^j(t)\right)+ H(\hat{\XX}^j\hat{u}^j(t)).
\end{align}
Since $\hat{\mathcal{L}}^j$ takes the same form as $\mathcal{L}$ we have preserved the assumption on $\mathcal{L}$ including \eqref{introduction: equation 1.6 of MM} and \eqref{introduction: equation 1.7 of MM}.

Let $h\in \RR^m$ be a small enough and then denote
\begin{align*}
    \Psi_{j,h}:=\exp_{x_j}(h\cdot \XX) \quad \text{on}\quad B_j.
\end{align*}
Here without loss of generality we assume that $\Psi_{j,h}$ are diffeomorphisms onto its image, else we could shrink $B_j$'s. 
Now, define the following conjugated map
\begin{align}
    \Theta_{j,h}:= \Phi_j^{-1}\circ \Psi_{j,h}\circ \Phi_j, \quad \Theta_{j,h}(0)=h.
\end{align}
Now, observe that 
\begin{align*}
    \Phi_j^*(u\circ \exp_{x_j}(h\cdot \XX))(t)= (u\circ \Psi_{j,h}\circ \Phi_j)(t)= \hat{u}(\Theta_{j,h}(t)).
\end{align*}
Then, we define the perturbed vector fields $\hat{X}_i^{j,h}$ as the pull back of $X_i$ along $\Theta_{j,h}$. i.e, 
\begin{align}
    \hat{X}_i^{j,h}:= (\Theta_{j,h})^* \hat{X}_i^j(t)= (d\Theta_{j,h})_t^{-1}\hat{X}_i^j(\Theta_{j,h}(t)), \quad \hat{\XX}^{j,h}=(\hat{X}_1^{j,h}, \cdots, \hat{X}_m^{j,h}), 
\end{align}
where $(d\Theta_h)_t$ is the derivative of $\Theta_h$ at $t$. Then, set
\begin{align}\label{perturbed operator}
    \hat{\mathcal{L}}^{j,h} w(t):=-\Tr\left(A(\hat{\XX}^{j,h}w(t))\hat{\XX}^{j,h*}\hat{\XX}^{j,h}w(t)\right)+ H(\hat{\XX}^{j,h}w(t)).
 \end{align}
 \eqref{perturbed operator} is the original operator $\hat{\mathcal{L}}^j$ translated by $\Theta_{j,h}$. Again, no extra terms appear in the perturbed operator since we express every perturbation in terms of the vector fields $X_i$. 

 Now, we will see that $\hat{u}_h$ is a viscosity subsolution to $\hat{\mathcal{L}}_h^j$ at $t=h$ given $\hat{u}^j$ is a viscosity subsolution to $\hat{\mathcal{L}}^j$ at $t=0$. 

 Let $\psi^j\in C^2$ touch $\hat{u}_h^j$ from above at $t=h$. Define $\hat{\psi}^j(s):=\psi(\Theta_{j,h}(s))$. Then, $s\mapsto \hat{u}^j(s)-\hat{\psi}^j(s)$ has a local max at $t=0$. Since, $\hat{u}^j$ is a subsolution to $\hat{\mathcal{L}}^j$ at $t=0$, 
 \begin{align*}
     \hat{\mathcal{L}}^j\hat{\psi}^j(0)\leq 0.
 \end{align*}
 By construction of $\hat{\XX}^{j,h}$, 
 \begin{align*}
     \hat{\mathcal{L}}^j\hat{\psi}^j(0)=\hat{\mathcal{L}}^{j,h}\psi^j(h), 
 \end{align*}
because $\hat{X}_i\hat{\psi}^j(0)=\hat{X}_i^h\psi^j(h)$ and likewise for second order terms. Hence $\hat{\mathcal{L}}^{j,h}\psi^j(h)\leq 0$, i.e, $\hat{u}_h^j$ is a viscosity subsolution at $h$.

So, to finish off the proof we will pick $h:=s=t-t_0$. Then let 
\begin{align*}
    \eta_{0} := 1-\sum_{j=1}^k \eta_j \quad (\text{so supp($\eta_0$) $\subseteq \Omega_{(1+4R)\epsilon}\backslash \bigcup_{j}B_j'$ })
\end{align*}

Then define 
\begin{align}
  \mathcal{L}^{j,\epsilon}:= (\Theta_{s,j}^{-1})^*(\hat{\mathcal{L}^{s,j}}).
\end{align}
Now, define 
\begin{align}
    \mathcal{L}^{\epsilon}u:= \eta_{0}(\mathcal{L}u) + \sum_{j=1}^k\eta_j \left( \mathcal{L}^{j, \epsilon} u\right).
\end{align}

This will give us that $u^{\epsilon}$ is a viscosity solution to the operator given by  $$\mathcal{L}_{x_0}^{\epsilon}:= (\Theta_{s}^{-1})^*(\hat{\mathcal{L}^s}),$$
i.e, the pull back of the perturbed operator $\hat{\mathcal{L}^h}$ in the $t-$space the diffeomoeprhism $\Theta_{s}^{-1}.$ Hence we are done. 
 
\end{proof}
\begin{remark}
    The points $\{x_1, \cdots x_k\}$ from \ref{point 3 approximating by viscosity solution} should be thought of as the elements of $\argmax_{\Omega_{(1+4R)\epsilon}}(u-v)$(this is a finite set since $\bar{\Omega}$ is compact); see proof \ref{final proof}.
\end{remark}
\section{Comparison principle for semi convex/concave solutions}\label{Comparison principle for semi convex/concave solutions}
In this section, we will prove Theorem \ref{Introduction: Theorem 1.1 of MM} by proving comparison principles of varying level of difficulty for viscosity super/sub-solutions of 
\begin{align}\label{Semi-convex solutions: equation 3.1 of MM}
    \mathcal{L}u=0-\Tr\left(A(\mathfrak{X}u)\mathfrak{X}^*\mathfrak{X}u\right)+ H(\mathfrak{X}u), 
\end{align}
with $A$ and $H$ as in \eqref{introduction: equation 1.6 of MM} and \eqref{introduction: equation 1.7 of MM}.   

Following the structure of the proof in \cite{MM}, we wil first assume the sub/super solutions are semi convex and semi concave and hence, they are differentiable at maximal points and the gradients are partially continuous. 
\begin{lemma}\label{semi convex solutions: lemma 3.1 of MM}
    If there exists $u,v\in C(\bar{\Omega})$ that are respectively semi convex and semi concave such that $\mathcal{L}u < 0 \geq \mathcal{L}v$ (resp $\mathcal{L}v \leq 0 < \mathcal{L} v$) in the viscosity sense near the points in $\argmax (u-v)$ and $u\leq v$ in $\partial \Omega$. Moreover, assume that $u$ and $v$ are twice differentiable at a point in $\argmax(u-v)$, then we have $u\leq v$ in $\partial \Omega$. 
\end{lemma}
\begin{proof}
    We follow proof by contradiction. Assume the contrary i.e. $\exists x\in \Omega$ such that $u(x)> v(x)$, i.e. $\exists x_0\in \bar{\Omega}$ such that 
    \begin{equation*}
        u(x_0) -v(x_0) = \underset{x\in \Omega}{\max} \ \{u(x)-v(x)\}>0.
    \end{equation*}
    Since $u\leq v$ in $\partial \Omega$ we have $x_0 \in \bar{\Omega}\backslash \partial \Omega$. Now, since we have assumed that $u, v$ are differentiable at the maximal points of $u-v$, we get that $u,v$ are differentiable at $x_0$. 

    Since, $x_0\in \argmax(u-v)$ we have $D^2 u(x_0)\leq D^2 v(x_0)$, which gives 
    \begin{align*}
        &\XX u(x_0) = \XX v(x_0) =: \xi_0\\
        &\XX^*\XX u(x_0) \leq \XX^* \XX v(x_0). 
    \end{align*}
    The given condition implies that there exists $\gamma>0$ such that $\mathcal{L}v-\mathcal{L}u \geq \gamma >0$, which together with the above gives
    \begin{align*}
        0&\leq \Tr (A(\xi_0) (\XX^*\XX v(x_0)- \XX^*\XX u(x_0)))\\
        &= \Tr (A (\XX v(x_0))\XX^*\XX v(x_0))- \Tr (A (\XX u(x_0))\XX^*\XX u(x_0))\\
        &= -\mathcal{L} v(x_0)+\mathcal{L} u(x_0)\leq -\gamma <0, 
    \end{align*}
which is a contradiction and hence completing the proof. 
\end{proof}
\begin{remark}
    Lemma \ref{semi convex solutions: lemma 3.1 of MM} is the first instance of the realization that the comparison theorem is a local property near the maximum points of $u-v$ and hence we do not need that $u$ and $v$ are viscosity solutions ion the entire domain $\Omega$. 
\end{remark}
In the rest of the section we will prove stronger versions of Lemma \ref{semi convex solutions: lemma 3.1 of MM} by relaxing the assumption to $\mathcal{L}u\leq 0 \leq \mathcal{L}v$. Similar to \cite{MM}, given a sub-solution $u$ we construct perturbations $u_{\lambda}$ for $\lambda>0$ small enough such that $u_{\lambda}$ are strict sub-solutions and satisfy the conditions of Lemma $3.1$.

Next, we will state two technical lemmas without proof as they follow the same proof as in \cite{MM}. 

\begin{lemma}\label{semiconvex solutions: lemma 3.2 of MM}\cite[Lemma 3.2]{MM}
    Given any $h\in C^{2}(\RR)$ with $h'\geq 1, h''\geq 0$ and $\omega:\Omega\to \RR$ twice differentiable at $x_0\in \Omega$, we have the following, 
    \begin{align}\label{semiconvex solutions: equation 3.5 of MM}
        \mathcal{L}(h\circ \omega)(x_0)\leq \frac{1}{\phi(1/h'(\omega))} \left[ \mathcal{L}\omega(x_0)-\frac{h''(\omega)}{h'(\omega)} \mathcal{E}(\mathfrak{X} \omega(x_0))\right],
    \end{align}
    where $\mathcal{E}$ is as in \eqref{introduction: equation 1.6 of MM} and $\phi: (0,1]\to (0,1]$ is as in \eqref{introduction: equation 1.7 of MM}.
\end{lemma}
\begin{lemma}\label{semiconvex-solutions: lemma 3.3 of MM} \cite[Lemma 3.3]{MM}
    Let $w_p(x)= w(x)+p\cdot x$ for any $p\in \RR^n$. If $w$ is twice differentiable at $x\in \Omega$ and $|p|\leq 1$, then we have 
    \begin{align}\label{semiconvex solutions: equation 3.6 of MM}
        \left|\mathcal{L}w_{p}(x)-\mathcal{L}w(x)\right|\leq c\left[ \omega_A(|p|)(|\mathfrak{X}\mathfrak{X}w(x)|+|p|)+ |A(\mathfrak{X}w(x))||p|+ \omega_H(|p|)\right],
    \end{align}
    for some constant $c=c(n, \|\sigma\|_{L^{\infty}}, \|D\sigma\|_{L^{\infty}})>0$.
\end{lemma}
We refer the reader to \cite{MM} for the proof of Lemma \ref{semiconvex solutions: lemma 3.2 of MM} and \ref{semiconvex-solutions: lemma 3.3 of MM}. 

\begin{lemma}\label{Semi convex solutions: Lemma 3.4 from MM}
Let $u,v\in C(\overline{\Omega})$ be such that $u\leq v$ in $\partial \Omega$ and $\mathcal{L}u\leq 0\leq \mathcal{L}v$ in $\Omega$ in viscosity sense. If $\mathfrak{X}u$(resp. $\mathfrak{X}v$) does not vanish at all maximal points of $u-v$, and $u$ and $v$ are respectively semi-convex and semi-concave in a neighborhood of maximal points of $u-v$, then $u\leq v$ in $\Omega$. 
\end{lemma}
\begin{proof}
    Assume the contrary. By adding small constants, we can regard $u<v$ in $\partial \Omega$. Therefore, without loss of generality, we can assume $u-v\leq -\epsilon<0$ in $\partial \Omega$ for a number $\epsilon>0$. Since, $u$ and $v$ are semi-convex and semi-concave near maximal points of $u-v$ we can assume that there exists $\delta>0$ such that these neighborhoods have diameter at least $\delta$ (we can do this since $\overline{\Omega}$ is compact). 

The contrary hypothesis implies that $u(x)>v(x)$ for some $x\in \Omega$ and since $u\leq v$ in $\partial \Omega$, hence the maximal points of $u-v$ are in the interior. Thus, $\argmax_{\Omega}(u-v)$, 
\begin{align*}
    u(y)-v(y)= \max_{\Omega}\{u(x)-v(x)\}=:M_0>0,
\end{align*}
and according to the given condition $\mathfrak{X}u(y)\neq 0$. Now, let $u_{\lambda}=h_{\lambda}(u)$ for $\lambda>0$, defined by 
\begin{align*}
    h_{\lambda}(u)= u+ \lambda(u-u_0)^2,
\end{align*}
where $u_0=\inf_{\Omega} \ u$, so that $h_{\lambda}'(u)=1+ 2\lambda(u-u_0)\geq 1$ and $h_{\lambda}''(u)=2\lambda>0$. Also, $h_{\lambda}\rightarrow id$ as $\lambda\rightarrow 0^+$ and we have 
\begin{align}\label{eq 3.7 from MM}
    \|u_{\lambda}-u\|_{L^{\infty}}\leq 4\lambda \|u\|_{L^{\infty}}^2, 
\end{align}
for any $\lambda>0$. For a sequence $x_{\lambda}\in \argmax(u_{\lambda}-v)$ such that $x_{\lambda}\rightarrow x_0$ up to possible subsequence as $\lambda\to 0^+$, we have $x_0\in \argmax(u-v)$. Since $\mathfrak{X}u(x_0)\neq 0$ we have 
\begin{align*}
    |\mathfrak{X}u(x_0)|\geq \theta\ \text{for some}\ \theta>0.
\end{align*}
Therefore, $\mathfrak{X}u(x_0)= h_{\lambda}'(u)\mathfrak{X}u(x_0)=(1+2\lambda (u-u_0))\mathfrak{X}u(x_0)\neq 0$ with $|\mathfrak{X}u_{\lambda}(x_0)|\geq \theta$. Without loss of generality let us also 
assume $\lambda$ is large enough such that $d(x_{\lambda}-x_0)<\delta/2$. Observe that \eqref{eq 3.7 from MM} gives 
\begin{align*}
\max_{x\in \Omega} (u_{\lambda}-v) \geq \max_{x\in \Omega} (u-v)- \|
u_{\lambda}-u\|_{L^{\infty}}\geq M_0- 4\lambda \|u\|^2_{L^{\infty}}>0,
\end{align*}
whenever $0<\lambda< M_0
/ 4\|u\|_{L^{\infty}}^2$. Furthermore, the fact that $u-v\leq -\epsilon< 0$ in $\partial \Omega$ along with \eqref{eq 3.7 from MM} in $\partial \Omega$ for any $0< \lambda< \epsilon/4\|u\|^2_{L^{\infty}}$. This implies the maximum is interior, i.e $x_{\lambda}\in \Omega$. 

As u (respectively $-v$) is semi-convex in a neighborhood say $U$(diameter at least $\delta$) of $x_0$, there exists $\Lambda>0$ such that $u+\frac{1}{2}\Lambda |x|^2$ is convex in $U$, hence it is locally Lipschitz and $\|\nabla u\|_{L^{\infty}(U)}\leq c\|u\|_{L^{\infty}(U)}$ for some $c= c(n, \delta)>0$ in compact subsets of $U$. Therefore, for a choice of 
\begin{align*}
    \Lambda'> \Lambda (1+4\|u\|_{L^{\infty}}) +2c^2\|u\|_{L^{\infty}}^2,
\end{align*}
it is not hard to verify that $u_{\lambda}+\frac{1}{2}\Lambda'|x|^2$ is also convex in $U$ for any $0<\lambda<1$. From now on we restrict our study to points inside $U$ unless otherwise specified. We also assume that $U$ is contained in coordinate patch of the manifold $M$.We have 
\begin{align*}
\Lambda'|\xi|^2> \Lambda (1+4\|u\|_{L^{\infty}(U)})|\xi^2|+ 2\|\nabla u\|_{L^{\infty}(U)}^2 |\xi^2|> \Lambda (1+ 2(u-u_0))|\xi|^2+ 2\lambda (\nabla u\cdot  \xi)^2,
\end{align*}
for any $\xi\in \RR^n$, and thereby $\Lambda' \mathbb{I}> \Lambda (1+2(u-u_0))\mathbb{I}\pm 2\lambda (\nabla u\otimes \nabla u)$ a.e. in $U$ in the sense of matrices. Therefore, at a.e. $x\in \Omega$ that are points of twice differentiability, we have 
\begin{align*}
    D^2u_{\lambda}(z)&= D^2u(z)(1+2\lambda (u(z)-u_0))+ 2\lambda (\nabla u(z)\otimes \nabla u(z))\\
    & \geq -\Lambda (1+2\lambda(u(z)-u_0))\mathbb{I}+ 2\lambda (\nabla u(z)\otimes \nabla u(z)) \geq -\Lambda' \mathbb{I};
\end{align*}
we conclude that $u_{\lambda}$ is also semi-convex inside $U$. Now, differentiability at maximal points with interior maxima at $x_{\lambda}$
 implies $\nabla u_{\lambda}(x_{\lambda})= \nabla v(x_{\lambda})$. Since $x_{\lambda}\to x_0$ and from partial continuity of the gradient, $\nabla u(x_{\lambda})\to \nabla u(x_0)$ as $\lambda \to 0^+$, there exists $\lambda_0= \lambda_0(n, \theta, \|u\|_{L^{\infty}}+ \|v\|_{L^{\infty}}, \delta)>0$ small enough, such that $c\omega(d(x_{\lambda}))\leq \theta/2$ for any $0<\lambda < \lambda_0$ where $\omega$ is the modulus of continuity of the gradient, $c=c(n, \|\sigma\|_{L^{\infty}})>0$ so that we have $\mathfrak{X}u(x_{\lambda})\neq 0 $ with 
 \begin{align*}
     |\mathfrak{X}u(x_{\lambda})|\geq \theta/2, \ \forall \ 0<\lambda<\lambda_0.
 \end{align*}
 Therefore, from \eqref{2.4 of MM} we have 
 \begin{align}\label{semiconvex solutions: 3.8 of MM}
     \mathcal{E} (\mathfrak{X}u(x_{\lambda})) \geq \frac{a_{\theta}/2}{ \phi(\theta/2 |\mathfrak{X}u(x_{\lambda})|)}=: e_{\theta}(\lambda) >0. 
 \end{align}
 However, $u_{\lambda}$ and $v$ might not be twice differentiable at $x_0$ or $x_{\lambda}$ or for any $\lambda>0$. Since $U$ is contained in a coordinate patch we can make sense of linear perturbations of the solutions as in Aleksandrov's theorem (\ref{Background: Aleksandrov's theorem}) as we can make sense of usual Euclidean dot product for points on $U$. Therefore, we can use Jensen's Lemma (Lemma \ref{Background: Jensen's lemma}) and Theorem \ref{Background: Aleksandrov's theorem}, to enable linear perturbations locally. Let $\psi_{k}$ be a smooth bump function supported near $x_{\lambda}$ supported in a ball in $\{x: x\in U, ||x-x_{\lambda}||<1/k\}$.
 
\begin{align}\label{semi-convex solutions: equation 3.9 of MM}
    u_{\lambda, p_{k}}= u_{\lambda}(x)+ p_{k} \cdot x \psi_{k}(x), \  v_{q_{k}}(x)= v(x) + q_{k} \cdot x\psi_{k}(x), \ \text{with}\ |p_{k}|+|q_{k}|\leq 1/k,
\end{align}
so that for any large $k\in \mathbb{N}$ large enough, there exists
\begin{align}\label{semiconvex solutions: 3.10 of MM}
    z_{\lambda, k} \in \argmax (u_{\lambda, p_k}- v_{q
    _k}) \ \text{with}\ |z_{\lambda, k}-x_{\lambda}|< 1/k.
\end{align}
($p_k\cdot x$ and $q_k\cdot x$ represents the usual Euclidean dot product) so that $u_{\lambda}$ and $v$ are twice differentiable at $z_{\lambda,k}$. Let us assume that $k$ such that $k> 4/\delta$ so that $z_{\lambda, k}\in U$. From \eqref{eq 3.7 from MM} and \eqref{semi-convex solutions: equation 3.9 of MM} and the fact that $z_{\lambda,k}\in U$ we get 
\begin{align*}
    \max_{x\in U}(u_{\lambda, p_{k}}- v_{q_{k}})&=  \max_{x\in \Omega}(u_{\lambda, p_{k}}- v_{q_{k}})\\&\geq \max_{x\in \Omega} (u_{\lambda}-v)- \left(\|u_{\lambda}-u_{\lambda, p_k}\|_{L^{\infty}}+\|v-v_{q_{q_k}}\|_{L^{\infty}}\right)\\
    &\geq \max_{x\in \Omega} (u_-v)- \|u_{\lambda}-u\|_{L^{\infty}}-\left(\|u_{\lambda}-u_{\lambda, p_k}\|_{L^{\infty}}+\|v-v_{q_{q_k}}\|_{L^{\infty}}\right)\\
    &\geq M_{0}-4\lambda\|u\|_{L^{\infty}}^2- \sup_{x\in \Omega} |x|/k,
\end{align*}
whenever $0<\lambda< M_0/ 8\|u\|_{L^{\infty}}^2$ and $k> 2\sup_{x\in \Omega} |x|/M_0$. The boundary behavior remains the same as $u_{\lambda, p_k}$ and $v_{q_k}$ agrees with $u_{\lambda}$ and $v_{\lambda}$ respectively near the boundary $\partial\Omega$ by definition.  

Now, the interior maximality of $z_{\lambda,k}$ implies that $\nabla u_{\lambda,p_k}(z_{\lambda,k})=\nabla_v{q_k}(z_{\lambda,k})$ and $D^2u_{\lambda,p_k}(z_{\lambda,k})\leq D^2 v_{q_{k}}(z_{\lambda, k})$ which together with leads to 
\begin{align}\label{semiconvex solutions: lemma 3.4: eq 3.11 of MM}
 \mathfrak{X}u_{\lambda, p_k} = \mathfrak{X}v_{q_k}(z_{\lambda, k})=:\xi_{\lambda,k}\ \text{and}\    \mathfrak{X}\mathfrak{X}^*u_{\lambda, p_{k}} (z_{\lambda, k}) \leq \mathfrak{X}\mathfrak{X}^*v_{q_k}(z_{\lambda,k}).
\end{align}
Since $u_{\lambda}$ and $v$ are twice differentiable at $z_{\lambda,k}$, we also have the differentiability of $u$ at $z_{\lambda,k}$. So, $\mathcal{L}u(z_{\lambda,k})\leq 0 \leq \mathcal{L}v(z_{\lambda,k})$. Furthermore, since $z_{\lambda, k}\to x_{\lambda}$ as $k\to \infty$ we have $\mathfrak{X}u(z_{\lambda,k})\to \mathfrak{X}u(x_{\lambda})$ as $k\to \infty$ from partial continuity of the gradient. Therefore, from continuity of $\xi \mapsto  A(\xi)$ and \eqref{semiconvex solutions: 3.10 of MM}, we can regard
\begin{align}\label{semiconevx solutions: 3.12}
    \left\| \mathcal{E}(\mathfrak{X}u(z_{\lambda,k}))- \mathcal{E}(\mathfrak{X}u(x_{\lambda}))\right\|\leq c \omega_{\lambda}(1/k)
\end{align}
for a sub-additive modulus $\omega_{\lambda}: [0,\infty)\to [0,\infty)$ with $\omega_{\lambda}(1/k)\to 0^+$ uniformly as $k\to \infty$ and constant $c=c(n,\|A\|_{L^{\infty}}, \Omega)>0$. Hence, using Lemma \ref{semiconvex solutions: lemma 3.2 of MM}
, \eqref{semiconvex solutions: 3.8 of MM} and \eqref{semiconevx solutions: 3.12} we get 
\begin{align}\label{semiconvex solutions: lemma 3.4: eq 3.13 of MM}
    \mathcal{L}u_{\lambda}(z_{\lambda, k})& =\frac{1}{\phi(1/h_{\lambda}'(u))} \left[ \mathcal{L} u(z_{\lambda,k}) -\frac{h_{\lambda''(u)}}{h_{\lambda'(u)}} \mathcal{E}(\mathfrak{X}u(z_{\lambda, k}))\right]\nonumber\\
    &\leq \frac{-h_{\lambda''(u)}\mathcal{E}(\mathfrak{X}u(z_{\lambda, k}))}{h_{\lambda'(u)} \phi(1/h_{\lambda}'(u))} \leq \frac{- h_{\lambda}''(u)\left[\mathcal{E}(\mathfrak{X}u(x_{\lambda}))-c \omega_{\lambda}(1/k)\right]}{ h_{\lambda}'(u) \phi(1/h_{\lambda}'(u))}\nonumber\\
    &\leq \frac{-h_{\lambda}''(u) [e_{\theta}(\lambda)-c\omega_{\lambda}(1/k)]}{h_{\lambda}'(u)\phi(1/h_{\lambda}'(u))}
 \leq \frac{-h_{\lambda}''(u) e_{\theta}(\lambda)/2}{h_{\lambda}'(u) \phi(1/h_{\lambda}'(u))}=: -\gamma_{\theta}()\lambda<0;
 \end{align}
where the last inequalities are ensured for large $k
\in \NN$ as given any $0< \lambda< \lambda_{0}$ fixed, their exists $k_{0}(\lambda)\in \NN$ such that $\omega_{\lambda}(1/k)< e_{\theta}(\lambda)/2c$ for all $k\geq k_{0}(\lambda)$. From semi-convexity of $u$ and semi-concavity of $v$ near $z_{\lambda, k}$, we can conclude that 
\begin{align*}
    -\Lambda \mathbb{I} \leq D^2u_{\lambda} (z_{\lambda,k}) \leq D^2v(z_{\lambda, k})\leq \Lambda \mathbb{I}. 
\end{align*}
Hence, from Lemma \ref{semiconvex-solutions: lemma 3.3 of MM}, \eqref{Semi convex solutions: Lemma 3.4 from MM}, \eqref{semi-convex solutions: equation 3.9 of MM} and \eqref{semiconvex solutions: 3.10 of MM} we get 
\begin{align}\label{semiconvex solutions: lemma 3.4:bounding by modulus of continuity}
    \max\{|\mathcal{L}v_{q_k}(z_{\lambda, k})-\mathcal{L}v(z_{\lambda, k})|, |\mathcal{L}u_{\lambda,p_k}(z_{\lambda, k})-\mathcal{L}u_{\lambda}(z_{\lambda, k})|\}\leq c\omega(1/k),
\end{align}
where $c= c(n,\|\sigma\|_{L^{\infty}}, \|D\sigma\|_{L^{\infty}}, \Lambda, \|A\|_{L^{\infty}}+\|H\|_{L^{
\infty}})$>0 and $\omega$ is a sub-additive modulus of continuity. For a fixed $0<\lambda<\lambda_0$ we can pick $k_1(\lambda)\in \NN$ such that $\omega(1/k)<\gamma_{\theta}(\lambda)/4c$ for all $k\geq k_1(\lambda)$. Now, using \eqref{semiconvex solutions: lemma 3.4:bounding by modulus of continuity}, \eqref{semiconvex solutions: lemma 3.4: eq 3.11 of MM}, \eqref{semiconvex solutions: lemma 3.4: eq 3.13 of MM} we get 
\begin{align*}
    0&\leq \Tr\left( A(\xi_{\lambda, k})(\mathfrak{X}\mathfrak{X}^*v_{q_k}(z_{\lambda, k})- \mathfrak{X}\mathfrak{X}^*u_{\lambda, p_k}(z_{\lambda, k}))\right)\\
    &= \Tr\left( A(\mathfrak{X}u_{\lambda, p_k}(z_{\lambda, k}))\right)- \Tr\left( A(\mathfrak{X}u_{\lambda, p_k}(z_{\lambda, k}))\right)\\
    &= -\mathcal{L} v_{q_k} (z_{\lambda, k})+\mathcal{L}u_{\lambda, p_k}(z_{\lambda, k})\leq -\mathcal{L}v(z_{\lambda, k})+ \mathcal{L} u_{\lambda}(z_{\lambda, k})+2 c \omega(1/k)\\
    &\leq -\gamma_{\theta}(\lambda)+ 2c\omega(1/k)\leq -\gamma_{\theta}(\lambda)/2<0.
\end{align*}
Hence, we have a contradiction. In this proof, we used the non-vanoishing of $\mathfrak{X}u$ at the maximal points of $u-v$. In the case where we have the non-vanishing of $\mathfrak{X}v$ at the maximal points of $u-v$ the argument is similar. In this case, too, we can derive an inequality lile \eqref{semiconvex solutions: lemma 3.4: eq 3.13 of MM} with $u_{\lambda}$ replaced by $v_{\lambda}:= v-\lambda(v-v_0)^2$ with $v_0=\inf_{\Omega}v$ and $0<\lambda<1/4\|v\|_{L^{\infty}}$ small enough. 
\end{proof}

 Now, we will remove the assumption that $\mathfrak{X}u$ and $\mathfrak{X}v$ doesn't vanish at the maximal points of $u-v$ unlike in Lemma \ref{Semi convex solutions: Lemma 3.4 from MM}. We will have to use the full strength of maximal subellipticity. First let us define the following domain. 

\begin{align}\label{semi-convex solutions: equation 3.14 of MM}
    \Omega_{\delta}:=\{ x\in \Omega: \rho(x,\partial \Omega)>\delta\}\ \forall \ \delta>0.
\end{align}

    In the next Proposition we will prove the main comparison theorem for semi-convex and semi-concave functions. The skeleton of the proof remains the same as \cite[Proposition 3.5]{MM} and hence we will try to maintain most of the notations and proof structure from there. However, the fact that we are not working in a group and have to deal with exponential map will make the arguments different towards the end of the proof. 

\begin{prop}\label{Proposition 3.5 of MM}
    Let $\mathcal{L}$ be as in \eqref{Semi-convex solutions: equation 3.1 of MM} and $u,v \in C(\bar{\Omega})$ be respectively semi-convex and semi-concave such that $u\leq v$ in $\partial \Omega$ and $\mathcal{L}u\leq 0\leq \mathcal{L}v$ in $\Omega$ at the maximal points of $u-v$ in viscosity sense, then $u\leq v$ in $\Omega$.  
\end{prop}

\begin{proof}
 We proceed by a proof by contradiction. Assume the contrary, i.e. $\max_{\Omega}(u-v)>0$ and since $u\leq v$ in $\partial \Omega$, the maxima are attained in the interior of $\Omega$. Thus, we have 
 \begin{align*}
     u(x_0)-v(x_0)=\max_{x\in \Omega}\{u(x)-v(x)\}=M_0>0,
 \end{align*}
 for an interior point $x_0\in \Omega$. For any $\delta\geq 0$ and $h,l\in \RR^n$ with $\|h\|,\|l\| < \delta$, let us denote the translations $u_h, v_l:\Omega_{\delta}\to \RR$ by 
 \begin{align*}
     u_{h}(x):= u(\exp_{x}(h\cdot \XX)), \ v_{l}(x):= u(\exp_{x}(l\cdot \XX))
 \end{align*}
 for $x\in \Omega_{\delta}$ and
 \begin{align}\label{semi-convex solutions: 3.15 of MM}
     M_{\delta}(h,l)&=\max_{x\in \Omega_{\delta}}\{u_{h}(x)-v_{l}(x)\}\ \nonumber\\ \calA_{\delta}(h,l)&=\{x\in \Omega_{\delta}: u_{h}(x)-v_{l}(x)= M_{\delta}(h,l)\}.
 \end{align}
 It is easy to see that $M_{0}(0,0)=M_0$ and $x_0\in A_{0}(0,0)$. Since $x_0\in \Omega_{\delta}$ we have $M_{\delta}(0,0)=M_0>0$ for any $0<\delta<\rho(x_0, \Omega_\delta)$. Hence, the maxima are in the interior of $\Omega_{\delta}$ and therefore for all $0<\delta'< \delta$ since $\Omega_{\delta}\subset \Omega_{\delta'}$. Also, $M_{\delta}(0,0)=M_{\delta'}(0,0)=M_0>0$ since $x_0$ is in the interior of $M_{\delta'}$. Also, for some $h,l\in B_{\delta}(0)$ if $\calA_{\delta}(h, 0)\neq\emptyset\neq  \calA_{\delta}(0,l)$ then the corresponding maxima are in the interior of $\Omega_{\delta}$ and therefore $M_{\delta}(h,0)= M_{\delta'}(h,0)$ and $M_{\delta}(0,l)=M_{\delta'}(0,l)$ for all $0<\delta<\delta'$. Let us denote
 \begin{align}\label{semi-convex solutions: eq 3.16 of MM}
     \mathcal{A}= \bigcup_{\delta>0}\left(\bigcup_{h,l\in B_{\delta}(0)} \mathcal{A}_{\delta}(h,l)\right),
 \end{align}
 is contained in a compact subset since $\Omega$ is bounded and from \eqref{semi-convex solutions: equation 3.14 of MM}, $\Omega_{\delta}\neq \emptyset$ for $0\leq \delta\leq \text{diam}(\Omega)$. Without loss of generality assume that 
 \begin{align}\label{semi-convex solutions: equation 3.17 of MM}
     u(z)-v(z)\leq -\tau<0, \ \forall \ z\in \partial \Omega,
 \end{align}
 for any arbitrarily small $\tau>u$ by addition of an appropriate fixed constant to $u$ and $v$. We will also make an explicit choice of constants $c_1, c_2$ later on such that 
 \begin{align}\label{Semi convex/concave solutions: adding constants}
     u\mapsto \tilde{u}:=u+c_1\ \text{and}\ v\mapsto\tilde{v}:= v+c_2.
 \end{align}
It is important to note that as long as we pick constants such that $\tilde{u}<\tilde{v}$ on $\partial\Omega$ holds the rest of the arguments will also holds since the semi convexity/concavity, gradients and maximal sets are invariant under such relabeling as $\XX\tilde{u}=\XX\tilde{v}$ and $\argmax(\tilde{u}_h-\tilde{v}_h)= \argmax(u_h-v_h)$.

 We study the behavior of propagation of the maximal with respect to the translations now. From semi-convexity of $u$ and $-v$, we know that they are locally Lipschitz and $\|\mathfrak{X}u\|_{L^{\infty}}\leq c\|u\|_{L^{\infty}}$ and $\|\mathfrak{X}v\|_{L^{\infty}}\leq c\|v\|_{L^{\infty}}$ for some $c=c(n,\|\sigma\|_{L^{\infty}}, \text{diam}(\Omega))>0$ in compact subsets of $\Omega$. Note that $h\mapsto M_{\delta}(h,l)$ is Lipschitz, since for $x\in \mathcal{A}_{\delta}(h,l)$ and $x'\in \mathcal{A}_{\delta}(h',l)$,
 \begin{align}\label{semi-convex solution: equation 3.19 A}
     M_{\delta}(h,l)-M_{\delta}(h',l)=& u(\exp_{x}(h\cdot \XX))-v(\exp_{x}(l\cdot \XX))- u(\exp_{x'}(h'\cdot\XX))\nonumber\\&+v(\exp_{x'}(l\cdot \XX)) \nonumber\\ 
      \leq& u(\exp_{x}(h\cdot \XX))-v(\exp_{x}(l\cdot \XX))- u(\exp_{x}(h'\cdot\XX))\nonumber\\&+v(\exp_{x}(l\cdot \XX))\nonumber\\
     =&u(\exp_{x}(h\cdot \XX))-u(\exp_{x}(h'\XX)),
 \end{align}
 we used the maximality at $x'$ to get $u(\exp_{x}(h'\cdot \XX))\leq u(\exp_{x'}(h'\cdot \XX))$. To bound the the above quantity we will use Gr\"onwall type bound. Set $h_s:= h'+s(h-h')$ and $\gamma(s):=\exp_{x}(h_s\cdot \XX)$. Then, 
 \begin{align}\label{Setting up Gronwall}
     \frac{d}{ds} u(\gamma(s))=\sum_{i}(h_i-h_i')X_{i}u(\gamma(s)),
 \end{align}
 so 
 \begin{align*}
     |u(\exp_{x}(h\cdot \XX))-u(\exp_{x}(h'\cdot \XX))| \leq \int_{0}^{1} \|h-h'\| |\XX u(\gamma(s))|\ ds\leq \|h-h'\|\|\XX u\|_{L^{\infty}(\Omega_{\delta})}, 
 \end{align*}
 as for small $\|h\|, \|h'\|$ $\gamma(s)\subset \Omega_{\delta}$. Hence, we get the RHS of \eqref{semi-convex solution: equation 3.19 A} is bounded by $C \|h-h'\|\|\XX u\|_{L^{\infty}}$, and hence
 \begin{align}\label{semi-convex solutions: equation 3.19 of MM}
       M_{\delta}(h,l)-M_{\delta}(h',l)\leq C \|h-h'\|\|\XX u\|_{L^{\infty}}, 
 \end{align}
where   the constant $C$ independent of $h, h'$. A symmetric inequality similar to \eqref{semi-convex solutions: equation 3.19 of MM} can be obtained using maximality in at $x$ provides the other direction and thereby the Lipschitz bound
 \begin{align}\label{semi-convex solutions: equation 3.20 of MM}
     |M_{\delta}(h,l) -M_{\delta}(h,l')|\leq \|h-h'\|\|\mathfrak{X}v\|_{L^{\infty}}.
 \end{align}
 Similarly, $l\mapsto M_{\delta}(h,l)$ is also a Lipschitz function and by arguing similarly as \eqref{semi-convex solutions: equation 3.19 of MM} using maximality in $\mathcal{A}_{\delta}(h,l), \mathcal{A}_{\delta}(h,l')$ and differentiability at maximal points, we can obtain
 \begin{align}\label{semi-convex solutions: equation 3.21 of MM}
     |M_{\delta}(h,l)-M_{\delta}(h,l')|\leq \|l-l'\|\|\mathfrak{X}v\|_{L^{\infty}}.
 \end{align}
 Now, we again proceed to consider two cases as in \cite[Proposition 3.5]{MM}, which is one of the main novel ideas in \cite{MM} to adapt the proof of Barles-Busca \cite{BB} and \cite{ACJ}.

 \textbf{Case 1:}\label{prop::Comparison principle for semi convex/concave solutions:: case 1}  There exists $0<\delta_0\leq 1/4\min\{ \rho(x_0, \partial \Omega), M_0/\|\mathfrak{X}v\|_{L^{\infty}}\}$ and $l_0\in \RR^n, \|l_0\|\leq \delta$, such that for all $h\in \RR^m$ with $\|h\|<\delta_0$, there exists $x_h\in \mathcal{A}_{\delta_0}(h,l_0)$ such that we have $\mathfrak{X}u(\exp_{x_h}(h\cdot \mathfrak{X}))=0$.

From the differentiability at maximal points $\mathfrak{X}u(\exp_{x_h}(h\cdot \XX))$ is well-defined. For $0<\delta_0<M_0/2\|\mathfrak{X}v\|_{L^{\infty}}$ and $\|l_0\|\leq \delta_0$, using \eqref{semi-convex solutions: equation 3.21 of MM} we have
\begin{align}\label{semi-convex solutions: equation 3.22 of MM}
    M_{\delta_0}(0,l_0)\geq M_0- \|l_0\|\|\mathfrak{X}v\|_{L^{\infty}}\geq M_0/2>0.
\end{align}
Now, consider the Taylor series formula 
\begin{align}\label{semi-convex solutions: Taylor series}
    u(\exp_{x_h}(h\cdot \XX))= u(x_h)+ (h\cdot \XX)u (x) +o (\|h\|).
\end{align}
Also, using BCH formula
\begin{align*}
    \exp_{x}(h\cdot \XX) =\exp_{\exp_{x}(h'\cdot \XX)} \left((h-h')\cdot\XX+\frac{1}{2} \sum_{i,j} h_j'(h_i-h_i')[X_j, X_i]+ R(h, h')\right),
\end{align*}
where the remainder $R(h,h')$ is a linear combination of higher order commutators with coefficients $O(|h-h'|(|h|+|h'|)+ |h-h'|^2)$. So, evaluating $u$ along this composed flow and Taylor exapnding in the flow time yields
\begin{align}\label{setting up for big O bound}
    u(\exp_{x})(h\cdot \XX)-u(\exp_{x}(h'\cdot \XX))= (h-h')\cdot \XX u(\xi) +O(|h-h'|(|h|+|h'|)) \|\XX u\|_{C^{0}(\Omega)}
\end{align}
for some $\xi$ on the trajectory connecting $\exp_{x}(h\cdot \XX)$ and $\exp_{x}(h'\cdot \XX)$ inside $\Omega$. For $|h|, |h'|$ small, the second term is absorbed into $X\|h-h'\|$, so again youg get $O(\|h-h'\|)$. 

Now, using the assumption we made in Case 1, and the maximality at $x_h\in \mathcal{A}_{\delta_0}(h,l_0)$ and $x_{h'}\in \mathcal{A}_{\delta_0}(h', l_0)$, together with differentiability at maximal points and \eqref{semi-convex solutions: Taylor series} we get

\begin{align*}
u(\exp_{x_h}(h\cdot \XX)) - &v(\exp_{x_h}(l_0\cdot \XX))\\&\geq  u(\exp_{x_{h'}}(h\cdot \XX)) - v(\exp_{x_{h'}}(l_0\cdot \XX))\\
    &= u(\exp_{x_{h'}}((h'+ h-h')\cdot \XX)) - v(\exp_{x_{h'}}(l_0\cdot \XX))\\
    &= u(\exp_{x_{h'}}(h'\cdot \XX))-v(\exp_{x_{h'}}(l_0\cdot \XX)) +\\& O(\|h-h'\|)\cdot(\XX u)(\xi)+ o(\|h-h'\|), 
\end{align*}
for any $h,h'\in B_{\delta_0}(0)$, and $\xi$ is as in \eqref{setting up for big O bound}. From \eqref{semi-convex solutions: 3.15 of MM} and the above we get
\begin{align*}
     M_{\delta_0}(h,l_0)\geq M_{\delta_0}(h',l_0)+o (\|h-h'\|).
\end{align*}

Since the inequality is symmetric with respect to $h$ and $h'$, we conclude that at points of differentiability of the function $h\mapsto M_{\delta_0}(h,l_0)$, we have $\mathfrak{X}M_{\delta_0}(h,l_0)=0$. \eqref{semi-convex solutions: equation 3.20 of MM} tells us that $h\mapsto M_{\delta_0}(h,l_0)$ is Lipschitz. For a Rademacher type theorem theorem for a real-valued function from a general doubling metric measure space, we refer the reader to \cite{CJ} ; also see \cite{PP} where the they prove Rademacher theorem for Carnot groups. 

Using Lemma \ref{Cheeger's result on a.e differentiability} we have 
\begin{align*}
    \mathfrak{X}M_{\delta_0}(h,l_0)=0, \quad \forall \ h\in B_{\delta}(0)\ a.e.
\end{align*}
and hence, the Lipschitz constant constant of $h\mapsto M_{\delta_0}(h,l_0)$ is zero, and hence the function $h\mapsto M_{\delta_0}(h,l_0)$ is constant in $B_{\delta_0}(0)$. Thus, we have 
\begin{align*}
    M_{\delta_0}(h,l_0)= M_{\delta_0}(0,l_0), \quad \forall \|h\|<\delta_0.
\end{align*}
Hence, for any $\tilde{x}_0\in \mathcal{A}_{\delta}(0,l_0)$ and $\|h\|<\delta<\delta_0<\rho(x_0, \partial \Omega)$, using the above, \eqref{semi-convex solutions: 3.15 of MM} and interior maximality at $\tilde{x}_0$ we have 
\begin{align*}
    u(\tilde{x}_0)- v(\exp_{\tilde{x_0}}(l_0\cdot \XX))&= M_{\delta}(0,l_0)=M_{\delta_0}(h,l_0)\\
    & =u(\exp_{x_h}(h\cdot \XX))- v(\exp_{x_h}(l_0\cdot \XX))\\
    &\geq u(\exp_{\tilde{x}_0}(h\cdot \XX))- v (\exp_{\tilde{x}_0}(l_0\cdot \XX)),
\end{align*}
leading to $u(\tilde{x}_0)\geq u(\exp_{\tilde{x}_0}(h\cdot \XX))$. Thus, we have a sub-solution $u$ with a local maximum at $\tilde{x}_0\in \Omega$, which can be converted to a non-negative maximum by adding a large enough positive constant to $u$. From Corollary \ref{Background: Corolllary 2.3 of MM}, $u(x)= u(\tilde{x}_0)$ for all $x\in B_{\delta}(\tilde{x}_0)$. Also, for all $\|h'\|<\delta$, the maximality at $\tilde{x}_0\in A_{\delta}(0,l_0)$ implies 
\begin{align}\label{semi-convex solutions: maximality at tilde{x}}
    u(\tilde{x}_0)-v(\exp_{\tilde{x}_0}(l_0\cdot \XX))&\geq u(\exp_{\tilde{x}_0}(h'\cdot \XX))-v(\exp(l_0\cdot \XX)\circ\exp(h'\cdot \XX)(\tilde{x}_0))\nonumber\\
    &= u(\tilde{x}_0)-v(\exp(l_0\cdot \XX)\circ\exp(h'\cdot \XX)(\tilde{x}_0)),
\end{align}
where $\exp(l_0\cdot \XX)\circ\exp(h'\cdot \XX)(\tilde{x}_0)$ is the composition of two exponential maps applied to $\tilde{x}_0$. When we compose more than one exponential maps we will use this notation from now on. From \eqref{semi-convex solutions: maximality at tilde{x}} we get $v(\exp(l_0\cdot \XX)\circ\exp(h'\cdot \XX)(\tilde{x}_0)) \geq v(\exp_{\tilde{x_0}}(l_0\cdot \XX))$, which also means that the super-solutions $v$ has a local minimum at $\exp_{\tilde{x}_0}(l_0\cdot \XX)$. By adding a negative constant and converting the super-solution $v$ to have a non-positive minimum at $\exp_{\tilde{x}_0}(l_0\cdot \XX)\in \Omega$, we can use Corollary \ref{Background: Corolllary 2.3 of MM} again to conclude that $v(x)= v(\exp_{\tilde{x}_0}(l_0\cdot \XX))$ in a neighborhood of $\exp_{\tilde{x}_0}(l_0\cdot \XX)$. Hence, $\{x\in \Omega: u(x)-v(x)=u(\tilde{x}_0)- v(\exp_{\tilde{x}_0}(l_0\cdot \XX))\}$ being both open and closed  and $\Omega$ being connected we get that it is the entire $\Omega$. Thus, for every $x\in\Omega$, we have $u(x)-v(x)- u(\tilde{x}_0)- v(\exp_{\tilde{x}_0}(l_0\cdot \XX))=M_{\delta_0}(0,l_0)$ from \eqref{semi-convex solutions: equation 3.22 of MM}, which contradicts $u\leq v$ in $\partial \Omega$.

\textbf{Case: 2} For any $0<\delta<\frac{1}{4}\min\{ \text{dist}(x_0, \partial\Omega), M_0/\|\XX v\|_{L^{\infty}}\}$ and any $l\in \RR^m$ with $\|l\|\leq \delta$ there exists $h_l\in \RR^m$ with $\|h_{l}\|<\delta$, such that for all $x\in\mathcal{A}_{\delta}(h_l,l)$ we have $(\XX u)(\exp_{x}(h_l\cdot \XX))\neq 0$. 

First, we prove the following claim. 

\textbf{Claim:} There exists $h\in \RR^m$ with $\|h\|<\delta$, such that for all $x\in \mathcal{A}_{\delta}(h,h)$ we have $\XX u ( \exp_{x}(h\cdot \XX))\neq 0$.

Notice that the set $\{\exp_{x}(h_l\cdot \XX): x\in \mathcal{A}_{\delta}(h_l,l)\}$, for any $l\in B_{\delta}(0)$ is contained in a compact set $K_{\delta}\subset \Omega_{\delta}$ which can be taken as the $\delta$-neighborhood of $\mathcal{A}$ as in \eqref{semi-convex solutions: eq 3.16 of MM}. Therefore, assuming case 2, we can regard 
\begin{align}\label{Semi convex/concave solutions: 3.23 of MM}
    |\XX u(\exp_{x}(h_{l}\XX ))|\geq \theta_{l}>0, \quad \forall x\in \mathcal{A}_{\delta}(h_l, l). 
\end{align}
Hence, let us take any $l_{0}\in B_{\delta}(0)$ and use the hypothesis of Case 2 repeatedly to define a sequence $l_{j+1}=h_{l_j}$ for every $j\in \NN\cup \{0\}$. Since $\{l_{j+1}\}$ is bounded, up to a sub-sequence we have $l_{j}\to h$ for some $h\in B_{\delta}(0)$ for some $h\in B_{\delta}(0)$ and hence $|l_j-h|, |l_{j+1}- l_j|\to 0^+$ as $j\to \infty$. We show that $h$ satisfies the claim. Indeed, as $\|h\|<\delta$, from \ref{semi-convex solutions: equation 3.20 of MM} and \ref{semi-convex solutions: equation 3.21 of MM}, we have
\begin{align*}
    M_{\delta}(h,h)\geq M_0-\|h\|(\|\XX u\|_{L^{\infty}}+ \|\XX v\|_{L^{\infty}}) \geq M_0/2,
\end{align*}
when $\delta< M_0/(4\|\XX u\|_{L^{\infty}})$. Now, for all $x\in \partial \Omega_{\delta}$, notice that 
\begin{align*}
    \dist(\exp_{x}(h\cdot \XX), \partial\Omega)\leq \dist(x, \partial \Omega)+ d(\exp_{x}(h\cdot\XX), x)\leq \delta+\|h\|<2\delta,   
\end{align*}
for $\|h\|<\delta$. Now, combining \ref{semi-convex solutions: equation 3.17 of MM} with the above inequality we get 
\begin{align}\label{3.24 of MM}
    u(\exp_{x}(h\cdot\XX))-v(\exp_{x}(h\cdot\XX))&\leq -\tau+\dist(\exp_{x}(h\cdot\XX), \partial\Omega) (\|\XX u\|_{L^{\infty}}+\|\XX v\|_{L^{\infty}})\nonumber\\
    &\leq -\tau +2\delta(\|\XX u\|_{L^{\infty}}+\|\XX v\|_{L^{\infty}})\leq 0,
\end{align}
if $\delta< \tau/(2\|\XX u\|_{L^{\infty}}+2\|\XX u\|_{L^{\infty}})$. Thus, we have $u_{h}\leq v_{h}$ on $\partial\Omega_{\delta}$, which implies the maxima at $M_{\delta}(h,h)>0$ is attained in the interior, hence $\mathcal{A}_{\delta}(h,h)\neq 0$. To prove the claim, we show that, given any $x\in \mathcal{A}_{\delta}(h,h)$ there exists $x_{j}'\in \mathcal{A}_{\delta}(h_{l_j}, l_{j})$ satisfying \ref{Semi convex/concave solutions: 3.23 of MM} which is close enough to $x$ for large $j$. Observe that by the definition of $\mathcal{A}_{\delta}(h,h)$, for any $\Omega'\subset\subset \Omega$ with $x\in \Omega'$, we have 
\begin{align*}
    u(\exp_{x}(h\cdot\XX))- v(\exp_{x}(h\cdot\XX))=\max_{\Omega'}(u_h-v_h).
\end{align*}
Since $\mathcal{A}_{\delta}(h,h)\neq 0$, we can see that $\mathcal{A}_{\delta}(l_j, l_j)$ is also nonempty for large enough $j$ because $l_{j}\to h$ as $j\to \infty$. Hence, assume that $\mathcal{A}_{\delta}(l_j, l_j)\neq \emptyset$ for $j\geq J_0$ for some $J\in \mathbb{N}$. For any $x\in \mathcal{A}_{\delta}(h,h)$ and $B\subset\subset \Omega$ with $x\in B$ let us denote $B_{j}:=\{\exp_{\exp_{y}(h\cdot \XX)}(-l_j \cdot \XX): y\in B\}$. Observe that given $x\in B$, $x\in B_{j}$ if $j\geq J$ and $|h-l_{j}|<\frac{1}{2}\dist(x, \partial B)$ and we have 
\begin{align}\label{3.25 of MM}
    M_{\delta}(l_j, l_j)&= \max_{y\in B}\{u(\exp_{y}(l_j\cdot \XX))- v(\exp_{y}(l_{j}\cdot \XX))\}\nonumber\\
    &=\max_{z\in B_j}\{u(\exp_{z}(h\cdot\XX))-v(\exp_{z}(h\cdot \XX))\}\nn\\
    &=u(\exp_{x}(h\cdot \XX))-v(\exp_{x}(h\cdot \XX))\nn\\
    &=u(\exp_{x_j}(l_j\cdot \XX))-v(\exp_{x_j}(l_j\cdot \XX)),
\end{align}
where $x_j:=\exp_{\exp_{x}(h\cdot \XX)}(-l_j\cdot\XX)$. Observe that $d(x_j, x)\leq |l_{j}-h|^{1/r}$ by Proposition \ref{Background: NSW lemma}. Also, by definition $x_j\in \mathcal{A}_{\delta}(l_j, l_j)$ and since $x\in B_j$ we have $x_j \in B$. 

We will now produce similarly a maximal point in $\mathcal{A}_{\delta}(h_{l_j}, l_j)$ close to $x$ by taking a choice of the relabeling as in \ref{Semi convex/concave solutions: adding constants}. Let 
\begin{align*}
    &c_{1,j}:= M_{\delta}(l_{j+1}, l_j)- M_{\delta}(l_{j+1, l_{j+1}}),\\
    &c_{2,j}:= v(\exp_{x_{j+1}}(l_{j+1}\cdot \XX))-v(\exp_{x_{j+1}}(l_{j}\cdot \XX)).
\end{align*}
Observe that the constants are invariant under relabelling by addtions of small constants to $u$ and $v$. Now, define
\begin{align}
    \tilde{u}:=u+c_{1,j}, \quad \tilde{v}:=v+c_{2, j}.
\end{align}
 Using \eqref{semi-convex solutions: equation 3.21 of MM}, \eqref{bound on cc metric after exp flow}, $|l_{j+1}-l_j|\leq |l_{j+1}-l_{j}|^{1/r}$ for $|l_{j+1}-l_{j}|\leq 1$, and the differentiability at maximal points we have 
\begin{align}\label{bound on constants c1j and c2j}
    |c_{1,j}|, |c_{2, j}|\leq C|l_{j+1}-l_{j}|^{1/r} \|\XX v\|_{L^{\infty}}
\end{align}
where $C$ does not depend on $l_{j+1}, l_j$ and $v$. Now, we can argue similar to \eqref{3.24 of MM} by using \eqref{bound on constants c1j and c2j} and \eqref{semi-convex solutions: equation 3.17 of MM} to get $\tilde{u}<\tilde{v}$ on $\partial \Omega$ for $j\geq J_1$ for some large enough $J_1\in \RR$. Hence we have made a relabelling that would agree with the previous arguments in the proof. 

Now, we can use \eqref{3.25 of MM} with $x_{j+1}=\exp_{\exp_{x}(h\cdot \XX)}(-l_{j+1}\cdot \XX)$ for $j\geq J_1$ to get 
\begin{align*}
    M_{\delta}(l_{j+1}, l_j)=&M_{\delta}(l_{j+1}, l_j)+c_{1, j}\\
    = &u\left(\exp_{x_{j+1}}(l_{j+1}\cdot \XX)\right)-v\left(\exp_{x_{j+1}}(l_{j+1}\cdot \XX)\right)+c_{1, j}\\
    =&u\left(\exp_{x_{j+1}}(l_{j+1}\cdot \XX)\right)-v\left(\exp_{x_{j+1}}(l_{j}\cdot \XX)\right)+\\
    &+u\left(\exp_{x_{j+1}}(l_{j}\cdot \XX)\right)-v\left(\exp_{x_{j+1}}(l_{j+1}\cdot \XX)\right)+c_{1, j}\\
    =& u\left(\exp_{x_{j+1}}(l_{j+1}\cdot \XX)\right)-v\left(\exp_{x_{j+1}}(l_{j}\cdot \XX)\right)+ c_{1, j}-c_{2, j}\\
    =&\tilde{u}\left(\exp_{x_{j+1}}(l_{j+1}\cdot \XX)\right)-\tilde{v}\left(\exp_{x_{j+1}}(l_{j}\cdot \XX)\right)
\end{align*}
Therefore, the relabelling $u\mapsto \tilde{u}$ and $v\mapsto \tilde{v}$ along with the above computation shows that
\begin{align}\label{xj+1 is in a maximal set}
    x_{j+1}\in \mathcal{A}_{\delta}(l_{j+1}, l_j), 
\end{align}
and \begin{align*}
    d(x_{j+1}, x_j)&=d\left(\exp_{\exp_x(h\cot \XX)}(-l_j\cdot \XX),\exp_{\exp_x(h\cot \XX)}(-l_{hj+1}\cdot \XX) \right)\\&\leq |l_{j+1}-l_j|^{1/r}.
\end{align*}
Therefore, for any $x\in \mathcal{A}_{\delta}(h, h)$, using partial continuity of gradient we can conclude that $\exists J_0=J_0(n, r, \delta, \|\XX u\|_{L^{\infty}})\in \NN$ large enough such that $\forall j\geq J_0$ we have 
\begin{align}\label{dominating by thetadelta}
    &\left|\XX u\left(\exp_{x}(h\cdot \XX)\right)-\XX u\left(\exp_{x_{j+1}}(h_{l_j})\cdot \XX\right)\right|\nonumber\\
    \leq &\left|\XX u\left(\exp_{x}(h\cdot \XX)\right)-\XX u\left(\exp_{x}(h_{l_j}\cdot \XX)\right)\right|+\nonumber\\
    &+\left|\XX u\left(\exp_{x}(h_{l_j}\cdot \XX)\right)-\XX u\left(\exp_{x_{j+1}}(h_{l_j}\cdot \XX)\right)\right|\nonumber\\
    \leq & c\left(\omega\left(d\left(\exp_{x}(h\cdot \XX), \exp_{x}(h_{l_j}\cdot \XX)\right)\right)+\omega\left(d\left(\exp_{x}(h_{l_j}\cdot \XX), \exp_{x_{j+1}}(h_{l_j}\cdot \XX)\right)\right) \right)\nonumber\\
    \leq & c\left(\omega(|h-h_{l_j}|^{1/r})+|h-h_{l_{j+1}}|^{1/r}\right)\nonumber\\
    \leq & \theta_{\delta/2},
\end{align}
where $\omega$ is a modulus dominating the modulus of (partial) continuity of the gradient. But we know that $x_{j+1}\in \mathcal{A}_{\delta}(l_j, l_{j+1})$ from \eqref{xj+1 is in a maximal set} and from \eqref{Semi convex/concave solutions: 3.23 of MM} we have 
\begin{align}\label{lower bound by thetadelta}
    \left| \XX u(\exp_{x_{j+1}}(h_{l_j}\cdot \XX))\right|\geq \theta_{\delta}>0.
\end{align}
Now, combining \eqref{lower bound by thetadelta} with \eqref{dominating by thetadelta} we get that $|\XX u(\exp_{x}(h_{l_j}\cdot \XX)|\geq \theta_{\delta}>0$ for any $x\in \mathcal{A}_{\delta}(h,h)$ and hence we are done proving the claim.

The proof idea in the rest of the proof is substantialy different from \cite[Proposition 3.5]{MM} as we don't have a concept of ``left invariance" as in the setting of a Lie group. 

However, we will use perturbation $\mathcal{L}^{h}$ of $\mathcal{L}$ as given by the construction in \eqref{point 3 approximating by viscosity solution} of Lemma \ref{approximating by convex solutions}. Now, $\mathcal{L}^h$ along with $u^h$ and $v_h$ will satisfy the hypotheses of Lemma \ref{Semi convex solutions: Lemma 3.4 from MM} and then taking $\delta\to 0$ will give us the necessary contradiction. 
\end{proof}


\begin{proof}[Proof of Theorem \ref{Introduction: Theorem 1.1 of MM}]\label{final proof} Let $u, v\in C(\bar{\Omega})$ be viscosity sub/super solutions of \eqref{Introduction: main equation} respectively with $u\leq v$ on $\partial \Omega$. Suppose for a contradiction assume that there exists $x_0$ in the interior of $\Omega$ such that 
\begin{align*}
    u(x_0)-v(x_0)=\max_{x\in \Omega}\{u(x)-v(x)\}>0.
\end{align*}
Without loss of generality also assume that $u-b\leq -\tau$ on $\partial \Omega$ for some small $\tau>0$. Then, from Proposition \ref{approximating by convex solutions}, we have $u^{\epsilon}$ and $v_{\epsilon}$ so that $\mathcal{L}^{\epsilon}u^{\epsilon}\leq 0\leq \mathcal{L}^{\epsilon} v_{\epsilon}$ at a given maximal point of $u^{\epsilon}-v_{\epsilon}$ in $\Omega_{(1+4R)\epsilon}$ in the viscosity sense for $R=2\max\{\|u\|_{L^{\infty}}, \|v\|_{L^{\infty}}\}$ and 
\begin{align*}
    \max\{\|u^{\epsilon}-u\|_{L^{\infty}}, \|v_{\epsilon}-v\|_{L^{\infty}}\}\leq c\omega(\epsilon), 
\end{align*}
for some modulus $\omega$ dominating the moduli of convergences. It is easy to see that $\Omega_{(1+4R)\epsilon}\to \Omega$ as $\epsilon\to 0^+$, and $u^{\epsilon}$ and $-v_{\epsilon}$ being semi convex are locally Lipschitz with $\|\XX u^{\epsilon}\|_{L^{\infty}}\leq c\|u^{\epsilon}\|_{L^{\infty}}$ and $\|\XX v_{\epsilon}\|_{L^{\infty}}\leq c\|v_{\epsilon}\|_{L^{\infty}}$ for some $c=c(n, \|\sigma\|_{L^{\infty}}, \text{diam}(\Omega))>0$ in compact subsets. Taking $\epsilon>0$ small enough so that $\omega(\epsilon)< \tau/4c$ and $([1+4R]\epsilon)^{1/2}(1+R)\leq \tau/2c$, we can conclude that for any $z\in \partial\Omega_{(1+4R)\epsilon}$
\begin{align*}
    u^{\epsilon}(z)-v_{\epsilon}(z)&\leq -\tau +2c\omega(\epsilon) +\text{dist} (x, \partial\Omega)(\|\XX u^{\epsilon}\|_{L^{\infty}}+\|\XX v_{\epsilon}\|_{L^{\infty}})\\
    &\leq -\tau/2 + c[(1+4R)\epsilon]^{1/2} (\|u^{\epsilon}\|_{L^{\infty}}+\|v_{\epsilon}\|_{L^{\infty}})\\
    &\leq -\tau/2 + cc[(1+4R)\epsilon]^{1/2} (1+R)\leq 0.
\end{align*}
Thus $u^{\epsilon}$ and $v_{\epsilon}$ along with $\mathcal{L}^{\epsilon}$(we takes the points $\{x_1, \cdots, x_k\}$ from \ref{point 3 approximating by viscosity solution} to be the elements in $\argmax_{\Omega_{(1+4R)\epsilon}}(u-v)$) satisfies the hypotheses of Proposition \ref{Proposition 3.5 of MM} in $\Omega_{(1+4R)\epsilon}$ and therefore $u^{\epsilon}\leq v_{\epsilon}$ in $\Omega_{(1+4R)\epsilon}$. Taking $\epsilon\to 0^+$ we get the required contradiction and hence the proof is complete.  
\end{proof}

 \bibliographystyle{alpha}
\bibliography{references}

\end{document}